\documentclass[]{article}

\usepackage[utf8]{inputenc} 
\usepackage{verbatim} %
\usepackage{graphicx}
\usepackage{amsfonts}
\usepackage{amsmath,amssymb}
\usepackage[all]{xy}
\usepackage{tikz}
\usetikzlibrary{arrows}
\usetikzlibrary{shapes}
\usepackage{tikz-cd}
\usetikzlibrary{calc,matrix,fit}
\usepackage[bookmarks = true, colorlinks=true, linkcolor = black, citecolor = black, menucolor = black, urlcolor = black]{hyperref}
\usepackage{latexsym}
\usepackage{authblk} 

\usepackage{mathtools} 
\usepackage{mathrsfs} 
\usepackage{eucal} 

\usepackage{enumerate}

\usepackage{amsmath}
\usepackage{amsthm}
	\theoremstyle{plain}
		\newtheorem{theorem}{Theorem}[section]
		\newtheorem{corollary}[theorem]{Corollary}
		\newtheorem{lemma}[theorem]{Lemma}
		\newtheorem{proposition}[theorem]{Proposition}
		\newtheorem{algorithm}[theorem]{Algorithm}
	\theoremstyle{definition}
		\newtheorem{definition}[theorem]{Definition}

	\theoremstyle{remark}
		\newtheorem{remark}[theorem]{Remark}

\usepackage{microtype}
\usepackage{mparhack}

\definecolor{light-gray}{gray}{0.9}
\definecolor{med-gray}{gray}{0.75}
\definecolor{dark-gray}{gray}{0.6}

\definecolor{lapislazuli}{rgb}{0.15, 0.38, 0.61}
\definecolor{crimson}{rgb}{0.86, 0.08, 0.24}
\newcommand{\blue}{\color{lapislazuli}\makeatletter\let\default\current\makeatother}
\newcommand{\red}{\color{crimson}\makeatletter\let\default\current\makeatother}
\definecolor{cadmiumgreen}{rgb}{0.0, 0.42, 0.24}
\newcommand{\green}{\color{cadmiumgreen}\makeatletter\let\default\current\makeatother}

\newcommand{\Z}{\mathbb{Z}}
\newcommand{\Zset}{\mathbb{Z}}
\newcommand{\Nset}{\mathbb{N}}

\renewcommand{\=}{\coloneqq}	 

\DeclareMathOperator{\id}{Id}
\DeclareMathOperator{\im}{Im} 
\DeclareMathOperator{\kr}{Ker} 
\DeclareMathOperator{\coker}{Coker} 

\DeclareMathOperator{\gf}{GenFlin}
\DeclareMathOperator{\TEZ}{TEZ}
\DeclareMathOperator{\sg}{sg}		

\newcommand{\verteq}{\rotatebox{90}{$\! =$}} 

\newcommand{\ql}{\textquotedblleft}
\newcommand{\qr}{\textquotedblright}

\newcommand{\rrdc}{\mbox{\,\(\Rightarrow\hspace{-9pt}\Rightarrow\)\,}}
\newcommand{\lrdc}{\mbox{\,\(\Leftarrow\hspace{-9pt}\Leftarrow\)\,}}
\newcommand{\lrrdc}{\mbox{\,\(\Leftarrow\hspace{-9pt}\Leftarrow\hspace{-5pt}\Rightarrow\hspace{-9pt}\Rightarrow\)\,}} 

\newcommand{\algitem}{\begin{list}{$\bullet$}{\setlength{\topsep}{0pt}}
  \setlength{\itemsep}{0pt}}

\begin{document}

\title{Computing higher Leray--Serre spectral sequences of towers of fibrations}

\author[1]{Andrea Guidolin}
\author[2]{Ana Romero}

\affil[1]{Basque Centre for Applied Mathematics, Bilbao, Spain}
\affil[2]{University of La Rioja, Logro\~no, Spain}

\date{} 

\maketitle

\begin{abstract} 
 The higher Leray--Serre spectral sequence associated with a tower of fibrations  represents a generalization of the classical Leray--Serre spectral sequence of a fibration. In this work, we present algorithms to compute higher Leray--Serre spectral sequences leveraging the effective homology technique, which allows to perform computations involving  chain complexes of infinite type associated with interesting objects in algebraic topology. In order to develop the programs, implemented as a new module for the Computer Algebra system Kenzo, we translated the original construction of the higher Leray--Serre spectral sequence in a simplicial framework  and studied some of its fundamental properties.
\end{abstract}

\section{Introduction}
\label{sec:intro}

In algebraic topology and homological algebra, spectral sequences are constructions arising in a quite natural way from a filtration of a chain complex, that is a collection $\{F_p C_*\}_{p\in \Z}$ of nested chain subcomplexes indexed over the integers.
In a recent work~\cite{matschke2013successive}, B. Matschke proposed a generalization of the notion of spectral sequence, formulating a theory which allows to construct a generalized spectral sequence, or \emph{spectral system}, from filtrations indexed over any partially ordered set. 
This notion is refined in \cite{Matschke2014a} in the special case of spectral systems arising from chain complexes filtered over the integers in several different ways, called \emph{higher spectral sequences}.
One of the main motivations of Matschke was to describe a mathematical object general enough to unify several spectral sequences which usually one would apply in succession. This is the case of the higher Leray--Serre spectral sequence associated with a tower of fibrations, one of the motivating examples of the work, which generalizes the classical Leray--Serre spectral sequence \cite{Ser51} relating the homology of the base, fiber and total space of a fibration. 

Handling spectral systems and higher spectral sequences may be daunting, since they look more technical than usual spectral sequences and they involve more complicated bookkeeping and manipulation of the filtration. On the other hand, computing them is highly desirable, as they are able to detect finer details then ordinary spectral sequences. In particular, higher Leray--Serre spectral sequences contain more information on the involved towers of fibrations than ordinary spectral sequences associated with each fibration of the tower. The overarching purpose of the present work is to allow the use of these finer details in practice, which is conditional on the possibility of explicitly computing higher Leray--Serre spectral sequences. 
 
Like ordinary spectral sequences, spectral systems and higher spectral sequences are not \emph{algorithms}, meaning that they cannot always be immediately computed from a filtration of a chain complex. In particular, the notion of spectral system and higher spectral sequence involves groups and differential maps which, despite being mathematically well-defined, in many situations cannot be computationally determined. This is the case, for example, that occurs when the considered chain complex is of infinite type, a quite common situation with chain complexes associated with interesting objects, for instance in algebraic topology. 

In this work, we present general algorithms to compute the higher Leray--Serre spectral sequence associated with a tower of fibrations. Because of the mentioned difficulties, in order to develop our algorithms and programs we use \emph{effective homology}~\cite{RS02,RS06}, a technique devised to compute the homology of complicated spaces, which allows to perform computations involving chain complexes of infinite type. 
A key concept for the technique of effective homology is that of a \emph{reduction} (Definition \ref{def:red}) between two chain complexes, which is a special type of chain equivalence. In short, if a chain complex is connected via a suitable sequence of reductions to a second one, the latter can be used in computations, for instance to determine homology. Differently from \cite{RS02,RS06}, we consider chain complexes endowed with generalized filtrations; by studying the interaction of the filtrations and the reductions we describe how the effective homology method can be applied to compute spectral systems and higher spectral sequences. In a previous work \cite{GR18} we presented general results and algorithms for successive reductions between filtered chain complexes, under the assumption that both the reductions and the filtrations (on all the involved chain complexes) are known, acknowledging that specific methods are required to compute many relevant spectral systems. In this work, we introduce a specific method for higher Leray--Serre spectral sequences, describing how the appropriate reductions and filtrations can be defined directly from the tower of fibrations. In order to make them computationally treatable, a main accomplishment of the present article is the introduction of a construction of higher Leray--Serre spectral sequences in a simplicial framework, which is more suitable for computational purposes than the original topological framework of \cite{matschke2013successive}. The study of this simplicial version of higher Leray--Serre spectral sequences yields in particular the guarantee of the correctness of our algorithms.
The programs described in this work have been implemented as a new module for the Computer Algebra system Kenzo~\cite{Kenzo} and are available at  \url{https://github.com/ana-romero/Kenzo-external-modules}. 

Let us outline the structure of the paper.
In Section \ref{sec:preliminaries} we present some fundamental notions and results that we use in the rest of the work. Section~\ref{sec:SSys_over_downsets} contains a general introduction to selected aspects of the theory of higher spectral sequences. In Section~\ref{sec:serre_spectral_systems} we describe our simplicial version of the higher Leray--Serre spectral sequence, introducing definitions of generalized filtrations (Definitions \ref{defn:2crpr-filtration} and \ref{defn:2tnpr-filtration}; see also Definitions \ref{defn:crpr-tnpr-filtration} and \ref{def:filtr_m_fibr}  in the following sections) of chain complexes associated with the tower of fibrations, a fundamental step to explicitly compute higher Leray--Serre spectral sequences. Sections \ref{sec:efhm_for_serrespectralsystems} and \ref{sec:m_fibrations} present the main ideas underlying our algorithms, as well as the description of how the effective homology technique allows them to deal with a wide range of situations that cannot be handled by standard algorithms.
The interaction between the generalized filtrations and the reductions of the effective homology technique (detailed in Propositions \ref{prop:rdct1-compatible}, \ref{prop:rdct2-compatible}, \ref{prop:maps_rho'1}, \ref{prop:maps_rho'2} and \ref{prop:rhoi_filtr}, \ref{prop:rho1_filtr}, \ref{prop:rho2_filtr}) yields our first main result (Theorems \ref{thm:2crpr-2page_S_effective} and \ref{thm:mcrpr-2page_S_effective}), which provides a theoretical guarantee of the correctness of our algorithms: from the $2$-page on (Definition \ref{def:2page_S_downsets}), the higher spectral sequence computed using our algorithms based on the effective homology method coincides with the higher Leray--Serre spectral sequence defined from the tower of fibrations. 
Section \ref{sec:second_page} contains the proof, in the simplicial setting we adopt throughout this work, of the second main result of this paper (Theorem \ref{thm:Serre-Matschke-simpl}), which identifies the $2$-page of higher Leray--Serre spectral sequences. In Section \ref{sec:examples_and_computations} we provide examples and computations which highlight some unique features of the programs we developed, and we end the paper with a section of conclusions and ideas for further work.

\section{Preliminaries}
\label{sec:preliminaries}

We devote this section to presenting preliminary notions and results that will be helpful to understand the rest of the work. After introducing the notion of spectral system, we recall some definitions on simplicial sets and principal fibrations, that we will need to introduce our simplicial version of the higher Leray--Serre spectral sequence. Then, we illustrate the effective homology technique and recall some results in homological perturbation theory that we will need to use effective homology in our context. We conclude the section presenting some algorithms to compute spectral systems in general situations, which will be useful to better understand the specific methods we developed for the higher Leray--Serre spectral sequence.

Throughout this work, we denote the composition of two maps $f:X\to Y$ and $g:Y\to Z$ by $g f$. If not specified otherwise, all the modules we consider are $\Z$-modules.

\subsection{Spectral systems}

Spectral systems are a generalization of the classical notion of spectral sequence~\cite{mccleary2001user} to the case of filtrations indexed over a poset. 

\begin{definition}
A partially ordered set or \emph{poset} $(I,\le)$ is a set $I$ endowed with a partial order $\le$.
\end{definition}

\begin{definition}
A \emph{chain complex} $C_\ast$ is a sequence of pairs $C_\ast=(C_n,d_n)_{n \in \Zset}$ where $C_n$ are abelian groups and $d_n:C_n \rightarrow C_{n-1}$ (the \emph{differential maps}) are group homomorphism of degree $-1$ such that $d_{n-1} d_n= 0 $, for all $ n\in \Zset$. The $n$-\emph{homology group} of the chain complex $C_\ast$ is defined as $H_n(C_\ast)\=\mathrm{Ker}\ d_n / \mathrm{Im}\ d_{n+1}$, for each $n\in \Zset$. We denote $H_*(C_\ast)\=\{H_n(C_\ast)\}_{n\in\Zset}$ the \emph{graded homology group} of the chain complex $C_\ast$. 
\end{definition}

In what follows, we will usually drop the subscript in the notation of differential maps. Homology groups can also be defined with coefficients in an arbitrary module $M$ rather than $\Z$, in which case we use the notation $H_n(C_*;M)$. For this definition and the relation with ordinary (integer) homology  see~\cite{maclane1963homology}.

\begin{definition}
\label{defn:i-filtration}
A filtration of a chain complex $C_\ast$ over a poset $(I,\le)$, briefly called an $I$-\emph{filtration}, is a collection of subcomplexes $F=\{F_i C_\ast \}_{i\in I}$ such that $F_i C_\ast \subseteq F_j C_\ast$ whenever $i\le j$ in $I$.
We call an \emph{$I$-filtered chain complex} $(C_* ,F)$ a chain complex $C_*$ endowed with an $I$-filtration $F=\{F_i C_*\}_{i\in I}$.
\end{definition}

We will often denote the chain subcomplexes $F_i C_\ast$ simply by $F_i$, forgetting about the grading of homology, when we are only interested in the filtration index $i$.

Now, we recall that for classical spectral sequences, which arise from a $\Zset$-filtration $\{F_p\}_{p\in \Zset}$, we have the formula (see~\cite{maclane1963homology}):
\begin{equation} \label{Epqr_classical}
E_{p,q}^r = \frac{Z_{p,q}^r + F_{p-1}C_{p+q}}{d(Z_{p+r-1, q-r+2}^{r-1}) + F_{p-1}C_{p+q}} ,
\end{equation}
where $Z_{p,q}^r \= \{ a\in F_p C_{p+q} : d(a)\in F_{p-r}C_{p+q-1} \}$. This expression can be rewritten in a form that highlights the interplay of the four filtration indices $p-r$, $p-1$, $p$ and $p+r-1$:
\begin{equation} \label{Epqr_classical2}
E_{p,q}^r = \frac{F_p \cap d^{-1}(F_{p-r})}{d(F_{p+r-1}) + F_{p-1}} ,
\end{equation}
where for simplicity we do not denote the grading of homology. To be precise, since in~(\ref{Epqr_classical2}) the denominator is not necessarily a subgroup of the numerator, the formula has to be interpreted keeping in mind that \emph{by convention} the notation $A/B$ for quotient groups will mean $A/(B\cap A) \cong (A+B)/B$. 

In~\cite{matschke2013successive}, this formula was imitated and 
generalized to the case of $I$-filtrations as follows.

\begin{definition}
\label{def:sp_sys_def1}
Let $(C_*,F)$ be an $I$-filtered chain complex.
Given a $4$-tuple of indices $z\le s\le p\le b$ in $I$ we define
\begin{equation}
\label{eq:S_term_def1}
S[z,s,p,b] \= \frac{F_p \cap d^{-1}(F_z)}{d(F_b)+F_s} .
\end{equation}
\end{definition}

As an example, in the particular cases $z=s\le p=b$ we can observe that $S[s,s,p,p]$ is the relative homology of $F_s \subseteq F_p$, that is $S[s,s,p,p]=H(F_p /F_s)$.

We call the collection $\{S[z,s,p,b]\}_{z\le s\le p\le b}$ the \emph{spectral system} associated with the $I$-filtration $F=\{ F_i\}_{i\in I}$ of $C_*$, and we call each abelian group $S[z,s,p,b]$ a \emph{term} of the spectral system.
For the sake of clarity, reintroducing in the notation (\ref{eq:S_term_def1}) the \emph{total degree} $n$ gives
\[ S_n[z,s,p,b] \= \frac{F_p C_n \cap d^{-1}(F_{z} C_{n-1})}{ d(F_b C_{n+1}) +F_s C_n}.
\]

The notion of differential in classical spectral sequences, as well as the way of obtaining terms of the page $r+1$ by taking homology at page $r$, can be generalized too.
Given two $4$-tuples of indices $z_1 \le s_1 \le p_1 \le b_1$ and $z_2 \le s_2 \le p_2 \le b_2$ in~$I$, it is easy to check that $d$ induces a well-defined differential
\[ d: S[z_2,s_2,p_2,b_2] \to S[z_1,s_1,p_1,b_1]
\]
whenever the additional inequalities $z_2 \le p_1$ and $s_2 \le b_1$ are satisfied. With a small abuse of notation, we denote with $d$ also the induced differentials. In the case of $z_2 = p_1$ and $s_2 = b_1$, a direct computation shows that the kernel and cokernel of $d: S[z_2,s_2,p_2,b_2] \to S[z_1,s_1,p_1,b_1]$ have the following nice expressions as terms of the spectral system:
\begin{align*}
     \kr d &=S[s_1,s_2,p_2,b_2] , \\
     \coker d &=S[z_1,s_1,p_1,p_2].
\end{align*}

The following result from \cite{matschke2013successive} describes how the homology of a sequence of such differentials can be expressed.

\begin{proposition}
\label{prop:homology_z1b3}
Consider three $4$-tuples of indices in $I$ satisfying the relations 
\begin{equation*}
\label{eq:inequalities_z1b3}
\begin{array}{ccccccccccccccc}
{} &{} &{} &{} &{} &{} &{} &{} &z_3 & \le &s_3 & \le &p_3  & \le &b_3 \\
{} &{} &{} &{} &{} &{} &{} &{} &\verteq &{} &\verteq  &{} &{} &{} &{} \\
{} &{} &{} &{} &z_2 & \le &s_2 & \le &p_2  & \le &b_2   &{} &{} &{} &{} \\
{} &{} &{} &{} &\verteq &{} &\verteq &{} &{} &{} &{} &{} &{} &{} &{} \\
z_1 & \le &s_1 & \le &p_1  & \le &b_1  &{} &{} &{} &{} &{} &{} &{} &{}
\end{array}
\end{equation*}
and the sequence of differentials between the corresponding terms:
\begin{equation}
\label{eq:ddSSS} 
S[z_3,s_3,p_3,b_3] \xrightarrow{d'} S[z_2,s_2,p_2,b_2] \xrightarrow{d} S[z_1,s_1,p_1,b_1] .
\end{equation}
Then the homology at the middle term is given by
\[  \frac{\kr d}{\im d'} =S[s_1,s_2,p_2,p_3] .
\]
\end{proposition}

Note that, as the differential $d$ of $C_*$ is a graded map of degree $-1$, so are the induced differentials between terms of the spectral system (with respect to the total degree). Then, for example, making explicit the total degree in~(\ref{eq:ddSSS}) one obtains $$S_{n+1}[z_3,s_3,p_3,b_3] \xrightarrow{d'} S_{n}[z_2,s_2,p_2,b_2] \xrightarrow{d} S_{n-1}[z_1,s_1,p_1,b_1],$$ whose homology at the middle term is $\kr d / \im d' =S_n[s_1,s_2,p_2,p_3]$.

The paper~\cite{matschke2013successive} introduces some 
examples of spectral systems (and higher spectral sequences, see Section \ref{sec:SSys_over_downsets}) associated with interesting objects in algebraic topology. For example, as we will explain in Section~\ref{sec:serre_spectral_systems}, the higher Leray--Serre spectral sequence is defined by means of a tower of fibrations, and generalizes the classical Leray--Serre spectral sequence of a fibration. 

However, the definitions in~\cite{matschke2013successive} are formal and the paper does not include a method to compute the spectral systems.

\subsection{Simplicial sets and fibrations}
\label{sec:simplicial_sets_and_fibrations}

In this section we introduce the definition of fibrations in a simplicial setting, following \cite{May67}.

\begin{definition}
\label{defn:smob}
Let $\mathcal{D}$ be a category. The category $s\mathcal{D}$ of \emph{simplicial objects in} $\mathcal{D}$ is defined as follows.
An object $K \in s\mathcal{D}$ consists of 
\begin{itemize}
\item for each integer $n \geq 0$, an object $K_n \in \mathcal{D}$;
\item for every pair of integers $(i,n)$ such that $0 \leq i \leq n$, \emph{face} and \emph{degeneracy} maps $\partial_i: K_n \rightarrow K_{n-1}$ and $s_i: K_n \rightarrow K_{n+1}$ (which are morphisms in the category~$\mathcal{D}$) satisfying the \emph{simplicial identities}:

\begin{align*}
\partial_i \partial_j &= \partial_{j-1} \partial_i && \text{if } i < j \\
s_i s_j &= s_{j+1} s_i && \text{if } i \leq j \\
  & && \text{if } i < j \\
  \partial_i s_j &= 
    \smash{\left\{\begin{array}{@{}l@{}}
      s_{j-1} \partial_i \\[\jot]
      \mathrm{Id} \\[\jot]
      s_j \partial_{i-1}
    \end{array}\right.} && \text{if } i=j, j+1 \\
  & && \text{if } i > j+1
\end{align*}
\end{itemize}
Let $K$ and $L$ be simplicial objects. A \emph{simplicial map} (or \emph{simplicial morphism}) $f: K \rightarrow L$ consists of maps $f_n: K_n \rightarrow L_n$ (which are morphisms in $\mathcal{D}$) which commute with the face and degeneracy operators, that is 
$f_{n-1} \partial_i = \partial_i  f_n$ and $f_{n+1} s_i = s_i f_n$ for all $0 \leq i \leq n$.
\end{definition}

If the objects of $\mathcal{D}$ have \emph{elements}, the elements of $K_n$ are called the \emph{$n$-simplices} of $K$. 

\begin{definition}
\label{def:deg_nondeg}
An $n$-simplex $x\in K_n$ is \emph{degenerate} if $x=s_jy$ for some $y \in K_{n-1}$ and some $0 \leq j < n$; otherwise $x$ is called \emph{non-degenerate}. An element $x\in K_n$ has \emph{degeneracy degree} equal to $p$, denoted $\deg x =p$, if $x = s_{i_{n-p}} \cdots s_{i_1} y$ for some non-degenerate $y \in K_p$.
\end{definition}

A \emph{simplicial set} is a simplicial object in the category of sets.
In what follows, we will use the notion of $1$-\emph{reduced} simplicial set, that is a simplicial set having a unique $0$-simplex and no non-degenerate $1$-simplex.
A \emph{simplicial group} $G$ is a simplicial object in the category of groups; in other words, it is a simplicial set where each $G_n$ is a group and the face and degeneracy operators are group morphisms.

A simplicial set $K$ has a canonically associated chain complex $C_*(K)=(C_n(K),d_n)$, where each chain group $C_n(K)$ is defined as the free $\Z$-module generated by $K_n$, and the differential $d_n : C_n(K) \to C_{n-1}(K)$ is defined as the alternating sum of faces, $d_n \= \sum_{i=0}^n (-1)^i \partial_i$.

\begin{definition} 
\label{def:Cartesian_KL}
The \emph{Cartesian product} $X\times Y$ of two simplicial sets $X$ and $Y$ is the simplicial set whose set of $n$-simplices is $(X\times Y)_n \= X_n \times Y_n$, with coordinate-wise defined face and degeneracy maps: if $(x,y)\in (X\times Y)_n$, then  
\begin{align*}
\partial_i (x,y) &\= (\partial_i x , \partial_i y), \qquad 0\le i\le n ; \\
s_i (x,y) &\= (s_i x , s_i y), \qquad 0\le i\le n .
\end{align*}
\end{definition}

\begin{remark}
\label{rmk:degCartesian}
From the definition of the degeneracy maps of a Cartesian product $X\times Y$, it is clear that an element $(x,y)\in (X\times Y)_n$ can be expressed as $(x,y) = s_{i_{n-q}} \cdots s_{i_1} (x',y')$ for some non-degenerate $(x',y') \in (X\times Y)_q$, where $s_{i_{n-q}}, \ldots , s_{i_1}$ are the \emph{common} degenerations in the expressions of $x$ and $y$ with respect to non-degenerate elements, as in Definition \ref{def:deg_nondeg}.
\end{remark}

\begin{definition} 
\label{def:Twisted_GB}
A \emph{twisting operator} from a simplicial set $B$ to a simplicial group $G$ is a map $\tau : B\to G$ of degree $-1$, that is a collection of maps $\tau = \{ \tau_n : B_n \to G_{n-1} \}_{n\ge 1}$, satisfying the following identities, for any $n\ge 1$ and for any $b\in B_n$:  
\begin{align*}
\partial_i ( \tau b) &= \tau (\partial_i b), \qquad 0\le i<n-1 , \\
\partial_{n-1} ( \tau b) &= \tau (\partial_n b)^{-1} \cdot \tau (\partial_{n-1} b), \\
s_i (\tau b) &= \tau (s_i b), \qquad 0\le i\le n-1 , \\
e_n &= \tau (s_n b) ,
\end{align*}
where $e_n$ is the identity element of $G_n$.
\end{definition}

We defined twisting operators in a slightly different (yet equivalent) way from \cite{May67}, in order to agree with the definition implemented in the Kenzo system.

\begin{definition}
\label{def:twcrpr}
Given a simplicial group $G$, a simplicial set $B$ and a twisting operator $\tau : B\to G$, 
the \emph{twisted (Cartesian) product} $E(\tau)\= G \times_{\tau} B$ is the simplicial set whose set of $n$-simplices is $E(\tau )_n = (G\times_{\tau} B)_n \= G_n \times B_n$ and whose face and degeneracy maps are defined in the following way: if $(g,b)\in (G\times_{\tau} B)_n$, then  
\begin{align*}
\partial_i (g,b) &\= (\partial_i g , \partial_i b), \qquad 0\le i < n , \\
\partial_n (g,b) &\= (\tau (b) \cdot \partial_n g , \partial_n b) ; \\
s_i (g,b) &\= (s_i g , s_i b), \qquad 0\le i\le n .
\end{align*}
\end{definition}

It can be easily shown that the identities defining a twisting operator $\tau$ are equivalent to the simplicial identities of $G\times_{\tau} B$. Notice that Remark \ref{rmk:degCartesian} applies also to twisted products $G\times_{\tau} B$.

\begin{definition}
\label{def:princ_fibration}
Given a simplicial group $G$, a simplicial set $B$ and a twisting operator $\tau : B\to G$, we call \emph{(principal) fibration} a sequence $G\xhookrightarrow{} E\to B$, where $E\= E(\tau ) = G\times_{\tau} B$ and the maps are the inclusion as first factor $G \xhookrightarrow{} G\times_{\tau} B$ and the projection on the second factor $G\times_{\tau} B 
\to B$. In a fibration, $B$ is called the \emph{base}, $G$ is called the \emph{fiber} and $E$ is called the \emph{total space}.
\end{definition}

\subsection{Effective homology}
\label{sec:efhm}

The effective homology method, introduced in~\cite{Ser94} and explained in depth in~\cite{RS02} and~\cite{RS06}, is a technique which can be used to determine the homology of complicated spaces, in particular spaces which are not of finite type. 
In a previous work~\cite{GR18}, we used this technique to develop algorithms and programs for computing some spectral systems, as we will briefly illustrate in Section~\ref{sec:programs}. We present now the main definitions and ideas of this method.

\begin{definition}
\label{def:red}
A \emph{reduction} $\rho\=(C_\ast \rrdc D_\ast)$ between two
chain complexes $C_\ast$ and $D_\ast$ is a triple $(f,g,h)$ where: (a) The components
$f$ and $g$ are chain complex morphisms $f: C_\ast \rightarrow D_\ast$ and $g: D_\ast \rightarrow C_\ast$; (b)
The component $h$ is a homotopy operator $h:C_\ast\rightarrow C_{\ast+1}$ (a graded group homomorphism of degree +1); (c) The following relations must be satisfied:
  (1) $f  g = \mbox{id}_{D_\ast}$; (2)
  $g f + d_{C_\ast} h + h  d_{C_\ast}
        = \mbox{id}_{C_\ast}$;
  (3)~\ {$f  h = 0$;} (4) $h   g = 0$; (5) $h   h = 0$.
\end{definition}

The relations of Definition \ref{def:red} express the fact that $C_\ast$ is the direct sum of $D_\ast$ and a contractible (acyclic) complex. This decomposition is simply $C_\ast=\kr  f \oplus \im  g$, with $\im  g\cong D_\ast$ and $H_n(\kr  f )=0$, for all $n$. In particular, this implies that the homology groups $H_n(C_\ast)$ and $H_n(D_\ast)$ are canonically isomorphic, for all $n$.

Given a chain complex $C_*$, the \emph{trivial reduction} $\id = (f,g,h) :C_* \rrdc C_*$ is the reduction with $f=g=\id$ and $h=0$. We now state two simple results (see \cite[Ch.~5]{RS06}) describing the behavior of reductions with respect to composition and tensor product.
\begin{proposition}
\label{prop:red_comp}
Let $\rho=(f,g,h): C_* \rrdc D_*$ and $\rho'=(f',g',h'): D_* \rrdc E_*$ be two reductions. Then a reduction $\rho''=(f'',g'',h''): C_* \rrdc E_*$ is given by:
\[ f'' \= f' f , \qquad g'' \= g  g' , \qquad h'' \= h + g  h'  f .
\]
\end{proposition}

\begin{proposition}
\label{prop:red_tensor}
Let $\rho=(f,g,h): C_* \rrdc D_*$ and $\rho'=(f',g',h'): C'_* \rrdc D'_*$ be two reductions. Then a reduction $\rho''\=\rho \otimes \rho' =(f'',g'',h''): C_* \otimes C'_* \rrdc D_* \otimes D'_*$ is given by:
\[ f'' \= f \otimes f' , \qquad g'' \= g \otimes g'  , \qquad h'' \= h \otimes \id_{C'_\ast} + (g  f) \otimes h'.
\]
\end{proposition}

\begin{definition}
A \emph{(strong chain) equivalence} $\varepsilon \= (C_\ast \lrrdc D_\ast)$ between two complexes $C_\ast$ and~$D_\ast$ is a triple $(\hat{C}_\ast,\rho,\rho ')$ where $\hat{C}_\ast$ is a chain complex and $\rho: \hat{C}_\ast \rrdc C_*$ and $\rho': \hat{C}_\ast \rrdc D_*$ are reductions: $C_\ast \lrdc \hat{C}_\ast \rrdc D_\ast.$
\end{definition}

\begin{definition}
An \emph{effective} chain complex $C_\ast$ is a free chain complex (that is, a chain complex consisting of free $\Zset$-modules) where each group $C_n$ is finitely generated, and there is an algorithm that returns a $\Zset$-base $\beta_n$ for each input degree $n$ (for details, see~\cite{RS02}). 
\end{definition}

Intuitively, an effective chain complex $C_\ast$ is a chain complex whose homology groups can be determined by means of standard algorithms for homology, based on matrix diagonalization and on the computation of the Smith Normal Form (see~\cite{KMM04}).

\begin{definition}
An \emph{object with effective homology} is a triple $(X,D_\ast,\varepsilon)$ where $X$ is an object (e.g. a simplicial set, a topological space) possessing a canonically associated free chain complex $C_\ast(X)$, $D_\ast$ is an effective chain complex and $\varepsilon = (C_\ast(X) \lrrdc D_\ast)$ is an equivalence between $C_\ast(X)$ and $D_\ast$.
\end{definition}

The notion of object with effective homology makes it possible to
compute the homology groups of \ql complicated\qr\ objects by using the associated effective complexes to perform the computations, which can be easily carried out via standard algorithms for homology. The method is based on the following idea: given some objects $X_1, \ldots, X_n$, a \emph{constructor} is an algorithm which produces a new object
$X$ (in Section~\ref{sec:pert_EZ} we will detail the case of the total space of a fibration as an example of a constructor). We assume that effective homology versions of the objects $X_1, \ldots , X_n$ are known, and we require that the constructor produces also an effective homology version of the space $X$. In this way, the effective chain complex $D_*$ associated with $X$ can be used for the computations, and the relationship with the original chain complex $C_*(X)$ is kept thanks to the equivalence $C_\ast(X) \lrrdc D_\ast$.   

The most significant achievement of the effective homology technique concerns the possibility to compute the homology of chain complexes of infinite type, which are not uncommon in algebraic topology and homological algebra, associated with interesting objects. 
We say that a chain complex $C_*$ is of \emph{infinite type} 
if at least one of its chain groups $C_n$ is not a finitely generated $\Z$-module. For brevity, we say that an object $X$ is of infinite type when the associated chain complex $C_*(X)$ is of infinite type. In this case, since one cannot save a list of all generators and express the differential maps as matrices, standard algorithms for homology are not directly applicable, and the effective homology technique is the only known way to deal with these situations.   

The effective homology method has been implemented in the system Kenzo~\cite{Kenzo}, a Common Lisp 16,000 lines program devoted to symbolic computation in algebraic topology, which has made it possible to determine homology and homotopy groups of complicated spaces and has proved its utility successfully computing some previously unknown results (for example, homology groups
of iterated loop spaces of a loop space modified by a cell attachment, components
of complex Postnikov towers; see~\cite{RS06} for details). 

\subsection{Homological perturbation and the (twisted) Eilenberg--Zilber reduction}
\label{sec:pert_EZ}

In the context of the present work, the homological perturbation results usually known by the names of Trivial Perturbation Lemma (TPL) and Basic Perturbation Lemma (BPL) turn out to be fundamental tools, since they describe how a perturbation (a modification of the differential of a chain complex) transmits through a reduction.

\begin{definition}
Let $C_*=(C_n,d_n)_{n \in \Zset}$ be a chain complex. A \emph{perturbation} $\delta$ of the differential $d$ is a family of morphisms $\delta=\{ \delta_n: C_n \to C_{n-1}\}_{n \in \Zset}$ such that the sum $d + \delta$ is again a differential, that is $(d + \delta)^2=0$ holds (meaning $(d_{n-1} + \delta_{n-1}) (d_n +\delta_n)= 0 $, for all $ n\in \Zset$).
\end{definition}

We call $C'_*=(C_n,d_n+\delta_n)_{n \in \Zset}$ the \emph{perturbed} chain complex obtained from~$C_*$ by introducing the perturbation $\delta$.

\begin{theorem} [Trivial Perturbation Lemma] 
\label{thm:tpl}
Let $C_*=(C_n,d_{C_n})_{n \in \Zset}$ and $D_*=(D_n,d_{D_n})_{n \in \Zset}$ be two chain complexes, $\rho=(f,g,h): C_* \rrdc D_*$ a reduction, and $\delta_D$ a perturbation of the differential
$d_D$. Then a reduction $\rho'=(f',g',h'): C'_* \rrdc D'_*$ exists, where:
\begin{enumerate}
\item$C'_* = (C_* , d_{C}+ g  \delta_D  f)$ is the perturbed chain complex obtained from $C_*$ by introducing the perturbation $g  \delta_D  f$;
\item $D'_* = (D_* , d_D +\delta_D )$ is the perturbed chain complex obtained from $D_*$ by introducing the perturbation $  \delta_D$; 
\item the maps of the new reduction $\rho'=(f',g',h')$ are given by $f'\= f$, $g'\= g$, $h'\= h$.
\end{enumerate}
\end{theorem}

\begin{theorem}[Basic Perturbation Lemma,~\cite{Brown1967}] 
\label{thm:bpl}
Let $C_*=(C_n,d_{C_n})_{n \in \Zset}$ and $D_*=(D_n,d_{D_n})_{n \in \Zset}$ be two chain complexes, $\rho=(f,g,h): C_* \rrdc D_*$ a reduction, and $\delta_C$ a perturbation of the differential
$d_C$. Suppose that the composition $h  \delta_C$ satisfies the following \emph{nilpotency condition}: for every $x \in C_*$ there exists a non-negative integer $m=m(x) \in \Nset$ such that
$(h \delta_C)^m(x)=0$. Then a reduction
$\rho'=(f',g',h'): C'_* \rrdc D'_*$ exists, where:
\begin{enumerate}
\item $C'_*= (C_* , d_{C}+  \delta_C )$ is the perturbed chain complex obtained from $C_*$ by introducing the perturbation $  \delta_C$;
\item $D'_*= (D_* , d_{D}+  \delta_D )$ is the perturbed chain complex obtained from $D_*$ by introducing the perturbation $\delta_D \= f  {\delta_C}  \varphi  g =
               f  \psi  {\delta_C}  g$;
\item the maps of the new reduction $\rho'=(f',g',h')$ are given by
\begin{equation*}
f'  \= f  \psi , \qquad
g'  \= \varphi  g , \qquad
h'  \= \varphi  h = h  \psi ,
\end{equation*}
\end{enumerate}
with the operators $\varphi$ and $\psi$ given by
\begin{equation*}
\varphi  \=\sum_{i=0}^{\infty}{(-1)^i(h \delta_C)^i} , \qquad
\psi  \=\sum_{i=0}^{\infty}{(-1)^i(\delta_C  h)^i} ,
\end{equation*}
the convergence of these series being guaranteed by the nilpotency condition.
\end{theorem}

Let us devote the rest of this section to effective homology of the twisted product of simplicial sets, which will be particularly relevant in what follows.
In order to state the definitions and the result correctly, we henceforth  assume all the chain complexes associated with simplicial sets to be \emph{normalized} (see \cite[Ch.~5]{May67}), which intuitively means that only non-degenerate simplices are considered as generators of the chain groups.
More precisely, for a simplicial set $X$ with the notation $C_*(X)$ we actually mean the quotient $C_*(X) / C_*^{\mathsf{D}}(X)$, where $C_*^{\mathsf{D}}(X)$ is the subcomplex of degenerate elements.
Let $G\rightarrow E \rightarrow B$ be a fibration given by a twisting operator $\tau: B\rightarrow G$;
we want to consider a constructor which produces the total space of the fibration, $E=G \times_\tau B$. Let us suppose that $B$ is $1$-reduced and that $G$ and $B$ are objects with effective homology, that is, there exist two equivalences $C_*(G) \lrrdc DG_*$ and $C_*(B) \lrrdc DB_*$, with $DG_*$ and $DB_*$ effective chain complexes. Then it is possible to obtain the effective homology of the total space $E$. 
The starting point is the Eilenberg--Zilber reduction, which describes the relation between the chain complex of a Cartesian product of simplicial sets $C_* (G\times B)$ and the tensor product $C_*(G)\otimes C_*(B)$.

\begin{theorem}[Eilenberg--Zilber,~\cite{Eilenberg1953a}]
\label{thm:EZ}
For any simplicial sets $G$ and $B$ there exists a reduction
\[ \rho = (f,g,h):  C_* (G\times B) \rrdc C_*(G)\otimes C_*(B) . 
\]
The maps $f,g,h$, which are called respectively the \emph{Alexander-Whitney}, \emph{Eilenberg--MacLane} and \emph{Shih maps}, are defined as follows:
\begin{align*}
&f(x_n ,y_n) \= \sum_{i=0}^n \partial_{i+1}\cdots \partial_n x_n \otimes \partial_0 \cdots \partial_{i-1} y_n ,   \\
&g(x_p \otimes y_q) \= \hspace{-10pt} \sum_{(\alpha, \beta)\in \{ (p,q)\text{-shuffles} \}} \hspace{-10pt} (-1)^{\sg (\alpha ,\beta )} (s_{\beta_q}\cdots s_{\beta_1}x_p , s_{\alpha_p} \cdots s_{\alpha_1} y_q)  , \\
&h(x_n, y_n) \= \hspace{-20pt} \sum_{\substack{0\le q \le n-1 , \\ 0\le p\le n-q-1 , \\ (\alpha , \beta) \in  \{ (p+1,q)\text{-shuffles} \}}} \hspace{-20pt} (-1)^{n-p-q+\sg (\alpha ,\beta )} 
\begin{aligned}[t]
(&s_\beta  \partial_{n-q+1}\cdots \partial_n x_n    , \\ &s_\alpha \partial_{n-p-q}\cdots \partial_{n-q-1} y_n) ,
\end{aligned}
\end{align*}
where a $(p,q)$-shuffle $(\alpha, \beta)=(\alpha_1 ,\ldots ,\alpha_p ,\beta_1 ,\ldots ,\beta_q)$ is defined as a permutation of the set $\{ 0,1,\ldots ,p+q-1 \}$ such that $\alpha_i <\alpha_{i+1}$ and $\beta_j <\beta_{j+1}$, $\sg (\alpha ,\beta)\=\sum_{i=1}^p (\alpha_i -i-1)$, $s_\beta \= s_{\beta_q +n-p-q}\cdots s_{\beta_1 +n-p-q} s_{n-p-q-1}$ and $s_\alpha \= s_{\alpha_{p+1} +n-p-q} \cdots s_{\alpha_{1} +n-p-q}$.
\end{theorem}

As we are interested in studying the twisted product $G\times_{\tau} B$ rather than the Cartesian product $G\times B$, we recall that the only difference between them concerns the face operators (see Definition~\ref{def:twcrpr}).  Consequently, the chain complexes $C_* (G\times B)$ and $C_* (G\times_{\tau} B)$ have the same underlying graded group, but different differentials. The differential of $C_* (G\times_{\tau} B)$ can be seen as a perturbed version of the differential of $C_* (G\times B)$, where the perturbation is given, for any $(g,b) \in G_n \times B_n$, by
\begin{equation} \label{eq:pert_twist} \delta (g,b) \= (-1)^n \left[ (\tau (b) \cdot \partial_n g , \partial_n b) - (\partial_n g , \partial_n b) \right] .
\end{equation}
A major classical result, known as the twisted Eilenberg--Zilber theorem, is obtained applying the Basic Perturbation Lemma to the Eilenberg--Zilber reduction $\rho = (f,g,h): C_* (G\times B) \rrdc C_*(G)\otimes C_*(B)$. 

\begin{theorem}[Twisted Eilenberg--Zilber,~\cite{Brown1959}]
\label{thm:twistedEZ}
Let $B$ be a simplicial set, $G$ a simplicial group and $\tau : B\to G$ a twisting operator. Then there exists a reduction
\[ \rho' = (f',g',h') : C_* (G\times_{\tau} B) \rrdc C_* (G) \otimes_t C_* (B) ,
\]
where $C_* (G) \otimes_t C_* (B)$ is the perturbed chain complex obtained from $C_* (G) \otimes C_* (B)$ by introducing the perturbation \ql induced\qr\ (via the application of the Basic Perturbation Lemma) by $\delta$.
\end{theorem}

Now, from the effective homologies of $G$ and $B$, we can construct a new equivalence from the tensor product $C_*(G)\otimes C_*(B)$ to $DG_*\otimes DB_*$ (see Proposition \ref{prop:red_tensor}), and using again the TPL and BPL (with the perturbation to be applied to the differential of $C_*(G)\otimes C_*(B)$ to obtain the differential of $C_*(G)\otimes_t C_*(B)$) we construct an equivalence from $C_*(G)\otimes_t C_*(B)$ to a new twisted tensor product $DG_*\otimes_t DB_*$, which is an effective chain complex (see~\cite{RRS06,RS06} for details). Finally, the composition of the two equivalences is the effective homology of $E \= G \times_\tau B$. In Section \ref{sec:efhm_for_serrespectralsystems} we will use and generalize this construction.

\subsection{Programs to compute spectral systems}
\label{sec:programs}

In a previous work~\cite{GR18}, we developed a set of programs for computing spectral systems based on the effective homology technique and implemented in the Kenzo system. The programs work 
in a similar way to the method that Kenzo uses to determine homology groups of a given chain complex: 
if an $I$-filtered chain complex $C_\ast$ is of finite type, its spectral system can be determined
by means of diagonalization algorithms on some matrices. Otherwise, the effective homology $C_\ast \lrdc \hat{C}_\ast \rrdc D_\ast$ of the initial chain complex $C_\ast$ is used to determine the spectral system as follows.

Let $C_\ast \rrdc D_\ast$ be a reduction. If the chain complexes $C_\ast$ and $D_\ast$ are endowed with $I$-filtrations $F$ and $F'$ respectively, in~\cite{GR18} we proved that, under suitable assumptions on the reduction $\rho$, (some terms of) the spectral systems of $(C_\ast,F)$ and $(D_\ast,F')$, denoted with the letters $S$ and $S'$ respectively, are isomorphic. This allows to compute the 
spectral system of the chain complex $C_\ast$ by using, to perform the computations, the chain complex~$D_\ast$, which in our scenario is assumed to be of finite type. More concretely, the following results express the conditions that are necessary to ensure that the spectral systems of the $I$-filtered chain complexes $C_\ast$ and $D_\ast$ are isomorphic.

\begin{theorem} 
\label{thm_Ana_orderred_gen}
Let $\rho = (f,g,h): C_* \rrdc D_*$ be a reduction between the $I$-filtered chain complexes $(C_* , F)$ and $(D_* ,F')$, and suppose that $f$ and $g$ are compatible with the filtrations, that is, for all indices $i \in I$ one has $f(F_i) \subseteq F'_i$ and $g(F'_i) \subseteq F_i$. 
Then, given a $4$-tuple of indices $z\le s\le p\le b$ in $I$,  the map $f$ induces an isomorphism between the spectral system terms
\[ f^{z,s,p,b}:  S[z,s,p,b]\to S'[z,s,p,b]  
\]
whenever the homotopy $h: C_\ast \to C_{\ast +1}$ satisfies 
the conditions  
\begin{equation} \label{eq.incl_gen_sameF}
h(F_z) \subseteq F_s  \qquad \text{ and } \qquad
h(F_p)  \subseteq F_b   . 
\end{equation} 
\end{theorem}

\begin{corollary}
\label{coro:thm_Ana_h_comp}
Let $\rho = (f,g,h): C_* \rrdc D_*$ be a reduction between the $I$-filtered chain complexes $(C_* , F)$ and $(D_* ,F')$, and suppose that  the maps $f,g,h$ are compatible with the filtrations. 
Then the map $f$ induces isomorphisms
\[ f^{z,s,p,b}:  S[z,s,p,b]\to S'[z,s,p,b]  
\]
for any 4-tuple of indices $z\le s\le p\le b$ in $I$.
\end{corollary}

Making use of these results, in~\cite{GR18} we developed the following algorithms, implemented in the Kenzo system. 

\begin{algorithm} 
\label{alg:neff-spsq-gr} Computation of the terms of a spectral system. \hspace{2pt} \newline
\noindent \emph{Input:}
\begin{itemize}
\item a chain complex $C_\ast=(C_n,d_n)$ with effective homology $C_\ast \stackrel{\rho_1}{\lrdc} \hat{C}_\ast  \stackrel{\rho_2}{\rrdc} D_\ast$,
\item $I$-filtrations for $C_\ast$, $\hat{C}_\ast$ and $D_\ast$ such that all the maps of the reductions $\rho_1$ and $\rho_2$ are compatible with the filtrations,
\item elements  $z \le s \le p \le b$ in $I$.
\end{itemize}
\emph{Output:} a basis-divisors representation of the group $S[z,s,p,b]$ of the spectral system associated with the filtered chain complex $C_\ast$, that is to say, a list of combinations $(c_1, \ldots , c_{k+\beta})$ which generate the group, together with the list of non-negative integers $(a_1,\ldots,a_k,0,\stackrel{\beta}{\ldots},0)$, where $a_1,\ldots ,a_k$ are the torsion coefficients of $S[z,s,p,b]$ and $\beta$ is its rank. The list of \emph{divisors} can be seen as the list of the coefficients of the elements that appear in the denominator with regard to the list of combinations that generate the group.
\end{algorithm}

Notice that we have stated this algorithm under the hypotheses of Corollary~\ref{coro:thm_Ana_h_comp}, namely that all the maps of the reductions $\rho_1$ and $\rho_2$, including the homotopies $h_1$ and $h_2$, are compatible with the filtrations. Algorithm~\ref{alg:neff-spsq-gr} can be applied also if $h_1$ and $h_2$ are not compatible with the filtrations, but the correctness of the output is guaranteed only for the terms $S[z,s,p,b]$ satisfying the hypotheses of
Theorem~\ref{thm_Ana_orderred_gen}: $h_i(F_z)\subseteq F_s$ and $h_i(F_p)\subseteq F_b$ for both homotopies ($i=1,2$). The same remark applies to the following algorithm.

\begin{algorithm} 
\label{alg:neff-spsq-dffr} Computation of the differentials of a spectral system. \hspace{2pt} \newline
\noindent \emph{Input:}
\begin{itemize}
\item a chain complex $C_\ast=(C_n,d_n)$ with effective homology $C_\ast \stackrel{\rho_1}{\lrdc} \hat{C}_\ast  \stackrel{\rho_2}{\rrdc} D_\ast$,
\item $I$-filtrations for $C_\ast$, $\hat{C}_\ast$ and $D_\ast$ such that all the maps of the reductions $\rho_1$ and $\rho_2$ are compatible with the filtrations,
\item elements $z_1\le s_1\le p_1\le b_1$ and $z_2\le s_2\le p_2\le b_2$ in $I$ such that $z_2 \le p_1$ and $s_2 \le b_1$,
\item a class $a \in S[z_2,s_2,p_2,b_2]$, given by means of the coefficients $(\lambda_1,\ldots,\lambda_t)$ with respect to the basis $(c_1,\ldots,c_t)$ of the group $S[z_2,s_2,p_2,b_2]$ determined by Algorithm~\ref{alg:neff-spsq-gr}.
\end{itemize}
\emph{Output:}
the coefficients of the class $d(a)\in S[z_1,s_1,p_1,b_1]$ with respect to the basis $(c'_1,\ldots,c'_{t'})$ computed by means of Algorithm~\ref{alg:neff-spsq-gr}.
\end{algorithm}

To improve the efficiency of Algorithms~\ref{alg:neff-spsq-gr} and~\ref{alg:neff-spsq-dffr}, discrete vector fields~\cite{For98} can also be used (see~\cite{GR18} for details, where programs computing discrete vector fields which are compatible with generalized filtrations on chain complexes are presented).

The implementation of these algorithms required the definition of a new class called \texttt{GENERALIZED-FILTERED-CHAIN-COMPLEX} which inherits from the Kenzo class \texttt{CHAIN-COMPLEX} and includes two new slots:\footnote{Several Lisp technical components have been omitted.} 

{ \footnotesize \begin{verbatim}
(DEFCLASS GENERALIZED-FILTERED-CHAIN-COMPLEX (chain-complex)
  ((pos :type partially-ordered-set )
   (gen-flin :type (function (generator) 
                 list-of-filtration-indexes))))
\end{verbatim}
}
\normalsize

The first slot, \texttt{pos}, is the poset over which the generalized filtration is defined. The second slot, \texttt{gen-flin},  is a function which inputs a generator of the chain complex and returns a list of elements of \texttt{pos}. This list represents the \emph{generalized filtration index} of the element, defined as follows.

\begin{definition}
\label{defn:gen-filtration-index}
Given a generator $\sigma \in C_\ast$, we define the \emph{generalized filtration index} of $\sigma$, denoted $\gf(\sigma)$, as the set of all indices $i\in I$ such that $\sigma \in F_{i} - \bigcup_{t < i} F_t$.
\end{definition}

Given now $p \in I$, the group $F_p$ corresponds to the free module generated by the set of generators $\sigma$ of $ C_\ast$ such that   there exists $i \in \gf(\sigma)$ with $i \leq p$. This implementation of generalized filtered chain complexes by means of the generalized filtration index notion is also valid for chain complexes of infinite type. We refer the reader to~\cite{GR18} for further details on these programs and some examples of computations.

\section{Higher spectral sequences}
\label{sec:SSys_over_downsets}

In this section we study some ideas from \cite[\S 3]{matschke2013successive} which are relevant for our work. In particular, we focus our attention on the two types of posets  playing a prominent role in the context of spectral systems, namely $\Z^m$ and the poset of its downsets $D(\Z^m)$. For spectral systems over $D(\Z^m)$ we illustrate the notion of \emph{connection}, that is a way to relate quotients defined from the filtered chain complex $C_*$ to the homology $H_*(C_*)$ through a series of homology computations, isomorphisms and groups extensions.  

First of all, let us consider $\Z^m$  as the poset $(\Z^m ,\le)$ with the coordinate-wise order relation $\le$, defined as follows: $P=(p_1,\ldots ,p_m)\le Q=(q_1,\ldots ,q_m)$ if and only if $p_i\le q_i$, for all $1\le i\le m$.  
\begin{definition}
\label{def:downset_DZm}
A \emph{downset} of $\Z^m$ is a subset $p\subseteq \Z^m$ such that if $P\in p$ and $Q\le P$ in $\Z^m$ then $Q\in p$. We denote $D(\Z^m)$ the collection of all downsets of $\Z^m$, which is a poset with respect to the inclusion $\subseteq$. 
\end{definition}
\begin{figure}[ht]
\centering
\includegraphics[scale=0.7]{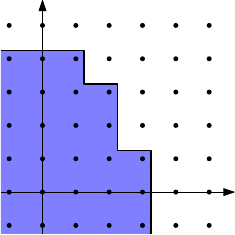}
\caption{The set of points of $\Z^2$ in the colored region is an example of a downset of $\Z^2$.}
\label{fig:ex_downset}
\end{figure}
%
Both $\Z^m$-filtrations and $D(\Z^m)$-filtrations arise in quite common situations. For example, when a chain complex $C_*$ is $\Z$-filtered in $m$ different ways, a $\Z^m$-filtration $\{ F_P\}_{P\in \Z^m}$ of $C_*$ can be easily defined (see \cite{matschke2013successive} for details). Canonically associated with a $\Z^m$-filtration $\{ F_P\}_{P\in \Z^m}$ there is a $D(\Z^m)$-filtration $\{F_p\}_{p\in D(\Z^m)}$ defined by setting, for each $p\in D(\Z^m)$,
\[ F_p \= \sum_{P\in p} F_P . 
\]

Spectral systems associated with $D(\Z^m)$-filtrations are the subject of the present section. 
\begin{figure}[ht]
\centering
\includegraphics[scale=0.7]{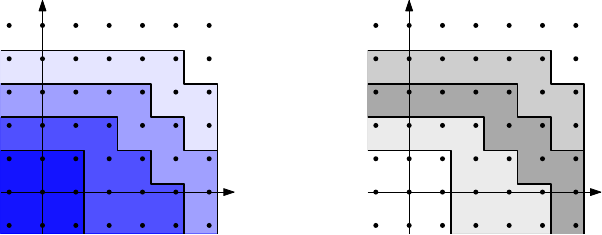}
\caption{Graphical representations of a term $S[z,s,p,b]$ of a spectral system over $D(\Z^2)$. On the left, the four downsets $z\le s\le p\le b$ in $D(\Z^2)$ defining the term $S[z,s,p,b]$ are colored in decreasing shades. On the right, the same term is represented following a convention we will adopt in what follows: the set difference $s\setminus z$ is colored in light gray \textcolor{light-gray}{$\blacksquare$}, $p\setminus s$ in dark gray \textcolor{dark-gray}{$\blacksquare$} and $b\setminus p$ in medium gray \textcolor{med-gray}{$\blacksquare$}. All the figures of this section are inspired by those of~\cite{matschke2013successive}.}
\label{fig:two_downsets}
\end{figure}
We call them \emph{higher spectral sequences}, adopting the terminology introduced in \cite{Matschke2014a}.
Although we have just illustrated a notable situation in which $D(\Z^m)$-filtrations can be defined, the results we will state apply to any $D(\Z^m)$-filtration $\{F_p\}_{p\in D(\Z^m)}$, with the additional hypothesis that it is \emph{distributive}: $F_{p\cap q}=F_p \cap F_q$ and $F_{p\cup q}=F_p + F_q$, for all $p,q \in D(\Z^m)$.
For instance, a $D(\Z^m)$-filtration is distributive if $C_*$ admits a decomposition $\bigoplus_{P\in \Z^m} C_P$ as graded abelian group and $F_p \= \bigoplus_{P\in p}C_P$ are chain subcomplexes, for all $p \in D(\Z^m)$. 

\begin{definition}
\label{def:connection}
Given a distributive $D(\Z^m)$-filtration of $C_*$, we call \emph{connection} any procedure which, starting with the $1$-page of the higher spectral sequence $\{S[z,s,p,b]\}$ over $D(\Z^m)$, determines $H_*(C_*)$ via a succession of homology computations, group extensions and natural isomorphisms between terms. In this context, we call \emph{$1$-page} of the higher spectral sequence $\{S[z,s,p,b]\}$ the set of terms $S_n[s,s,p,p]=H_n(F_p /F_s)$ such that $p$ \emph{covers} $s$, that is $s<p$ and there exists no $x\in D(\Z^m)$ with $s<x<p$.  With a small abuse of terminology, we will call \emph{$1$-page}  each subset of $\{S[s,s,p,p] \mid s,p\in D(\Z^m), \text{ $p$  covers  $s$} \}$ which can be seen as the starting point of a connection. 
\end{definition}
Notice that this agrees with the usual notion of $1$-page $\{E^1_p\}_{p\in\Z}=\{S[p-1,p-1,p,p]\}_{p\in\Z}$ for classical spectral sequences seen as part of a spectral system over $\Z$.

The essence of the study of connections for higher spectral sequences can be intuitively described as follows. First, one selects a suitable collection of downsets in $D(\Z^m)$ to employ as indices of terms of the higher spectral sequence, and uses Proposition~\ref{prop:homology_z1b3} to describe how computing homology affects the \ql shape\qr\ of the downsets. Then, applying a technical result \cite[Lemma 3.8]{matschke2013successive}, one can identify $4$-tuples of downsets with different shapes which determine the naturally isomorphic terms of the higher spectral sequence, a method which, if used properly, can allow to iterate the procedure. The use of Proposition~\ref{prop:homology_z1b3} limits our choice to $4$-tuples of downsets satisfying its hypotheses; for this reason,  the downsets we consider are often simply obtained as translations in $\Z^m$ of a single downset. 

Here we introduce the \emph{secondary connections} presented in \cite[\S 3.2]{matschke2013successive}, which play a prominent role in the generalization of the Leray--Serre spectral sequence. Even though a more general and flexible framework for this kind of connections was introduced by Matschke in~\cite{Matschke2014a}, for the purpose of this work we prefer the simpler and more explicit description of~\cite{matschke2013successive}, to which we address the interested reader also for details on other kinds of connections.
In Remarks \ref{rmk:higher_spseq_2}  and \ref{rmk:higher_spseq_m} we will show that the simplicial version of the higher Leray--Serre spectral sequence we introduce in this work carries over to the additional structure of~\cite{Matschke2014a}, pointing out however that the secondary connections of \cite[\S~3.2]{matschke2013successive} are ideal from a computational point of view for the application of the effective homology technique.

Having defined the $1$-page of a higher spectral sequence over $D(\Z^m)$, we want now to introduce a notion of \emph{$2$-page}, which generalizes the usual one for classical spectral sequences and will appear even more \ql natural\qr\ in light of the results we will present later. Secondary connections represent a way to connect the $1$-page to the $2$-page computing homology $m$ times; the $2$-page can be then connected to the homology $H_*(C_*)$ in different fashions, which we will not detail here.

For $1\le k\le m$ define the automorphism $\varphi_k : \Z^m \to \Z^m$ as the map sending $X=(x_1, \ldots ,x_m)$ to
\[  \varphi_k (X) \= \left(  x_{k+1}, x_{k+2} ,\ldots ,x_m , \sum_{i=1}^k x_i ,\sum_{i=2}^k x_i ,\ldots ,x_k  \right) .
\]
Let $\le_{\text{lex}}$ denote the \emph{lexicographic order} on $\Z^m$. 
For $P\in \Z^m$ and $1\le k\le m$ define the downset
\[ T^k_P \= \left\{ X\in \Z^m \mid \varphi_k (X) \le_{\text{lex}} \varphi_k (P)  \right\} .
\]
\begin{figure}[ht]
\centering
\includegraphics[scale=0.7]{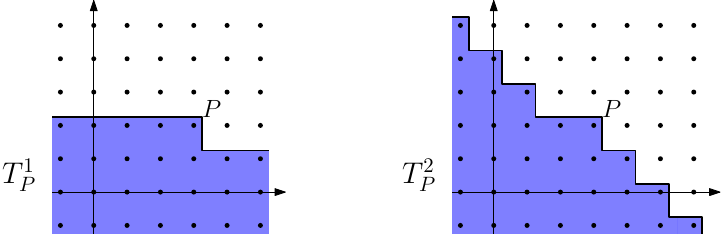}
\caption{Example of the downsets $T^1_P$ and $T^2_P$ in $\Z^2$ for $P=(3,2)$.}
\label{fig:13_TP}
\end{figure}
Let $e_i=(0,\ldots ,1,\ldots,0)$ be the element of $\Z^m$ whose only non-null entry is a~$1$ at position $i$, for $1\le i\le m$.  Given $P=(p_1,\ldots ,p_m)\in \Z^m$ and $1\le k\le m$ define the following downsets, where by convention $e_0 \= 0$ and $e_{-1} \= -e_m$:
\begin{equation}
\label{eq:TP_k}
\begin{aligned}
p(P;k) & \= T^k_{P}     \\
s(P;k) & \= p(P;k) \setminus \{ P\}  = T^k_{P+e_{k-1}-e_k}     \\
z(P;k) & \= p(P;k) -e_k =T^k_{P-e_k}     \\
b(P;k) & \= s(P;k) +e_k =T^k_{P+e_{k-1}}     \\
z^{*}(P;k) & \= z(P;k)\setminus \{ P-e_k \} =T^k_{P+e_{k-1}-2e_k}     \\
b^{*}(P;k) & \= b(P;k)\cup \{ P+e_k \} =T^k_{P+e_k} .    
\end{aligned}
\end{equation}
The $4$-tuples of downsets $z\le s\le p\le b$ and $z^{*} \le s\le p\le b^{*}$ define respectively the terms
\begin{equation}
\label{eq:defS,S*}
\begin{aligned}
S(P;k) &\= S[z(P;k),s(P;k),p(P;k),b(P;k)] , \\
S^{*}(P;k) &\= S[z^{*}(P;k),s(P;k),p(P;k),b^{*}(P;k)] .
\end{aligned}
\end{equation}
%
In this context, we consider as $1$-page of the higher spectral sequence the collection of the terms $S_n(P;1)=H_n(F_{p(P;1)}/F_{s(P;1)})=H_n(F_{p(P;1)}/F_{p(P;1)\setminus \{ P\} })$, for all $P\in \Z^m$.

\begin{definition}
\label{def:2page_S_downsets}
We call \emph{$2$-page} of a higher spectral sequence the collection of terms $S^{*}(P;m)$, for all $P\in \Z^m$.
\end{definition}

The following two lemmas describe a way to connect the terms $S(P;1)$ of the $1$-page to the terms $S^{*} (P;m)$ of the $2$-page.
\begin{lemma}[\cite{matschke2013successive}, Lemma 3.14]
\label{lem:differentials-e_k}
There exist differentials \emph{in direction $-e_k$},
\[ d: S(P;k) \to S(P-e_k;k) ,
\]
induced by the differential maps of $C_*$, such that the homology at the middle term of 
\[ S(P+e_k;k) \xrightarrow{d'} S(P;k) \xrightarrow{d} S(P-e_k;k)
\]
is $S^{*} (P;k)$.
\end{lemma}
\begin{figure}[ht]
\centering
\includegraphics[scale=0.55]{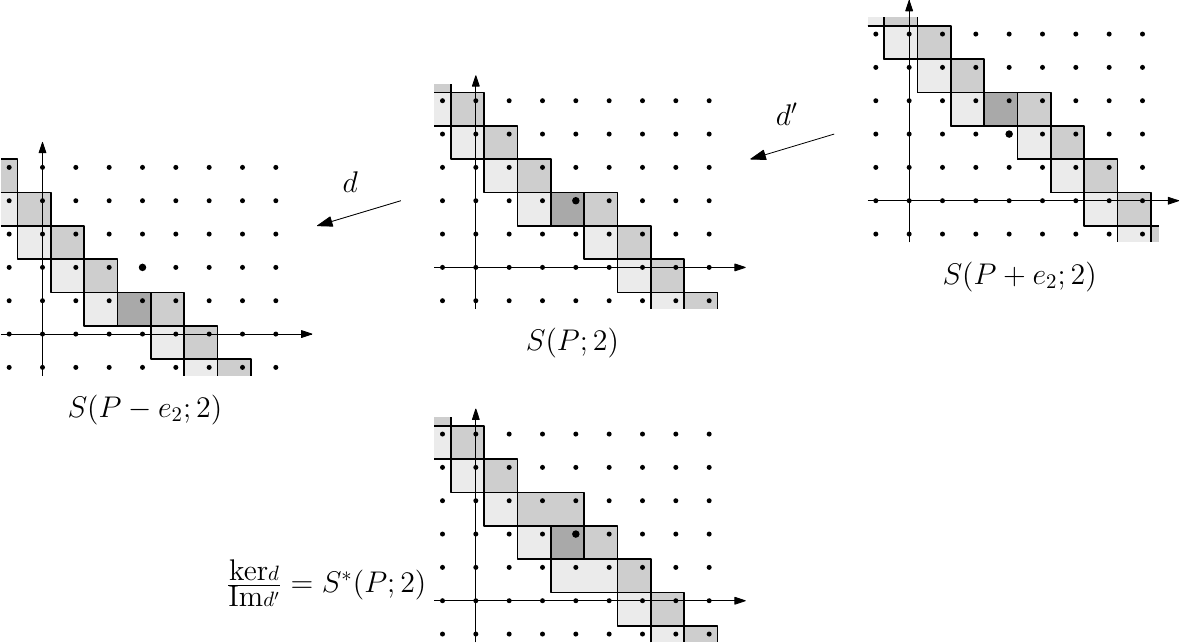}
\caption{Representation of Lemma~\ref{lem:differentials-e_k} for $m=k=2$. The differentials $d$ and $d'$ between terms of the $1$-page are in direction $-e_2=-(0,1)$, and the homology $\ker d / \im d'$ is a term of the $2$-page. The point $P=(3,2)$ is highlighted in the picture.}
\label{fig:ddS_secondary}
\end{figure}
Lemma~\ref{lem:differentials-e_k} is an application of Proposition~\ref{prop:homology_z1b3} (see Figure~\ref{fig:ddS_secondary}) to our current situation: a higher spectral sequence where some distinguished downset are defined by~(\ref{eq:TP_k}). The following result identifies terms of the higher spectral sequence which are naturally isomorphic:
\begin{lemma}[\cite{matschke2013successive}, Lemma 3.15]
\label{lem:isoSS*}
For any $1\le k\le m-1$ there is a natural isomorphism
\[ S^{*} (P;k)\cong S(P;k+1) .
\]
\end{lemma}

The core idea of secondary connections consists in using Lemma~\ref{lem:differentials-e_k} and Lemma~\ref{lem:isoSS*} alternately: starting from the $1$-page $\{ S(P;1)\}$ and taking homology in direction $-e_1$ one determines the terms $S^{*}(P;1)\cong S(P;2)$, then taking homology in direction $-e_2$ one determines the terms $S^{*}(P;2)\cong S(P;3)$; continuing alternating between homology and natural isomorphisms one eventually determines the $2$-page $\{ S^{*}(P;m)\}$. As we mentioned before, the $2$-page can then be connected to the homology $H_*(C_*)$ in different ways, for example using \emph{lexicographic} connections (we refer to~\cite{matschke2013successive} for further details).   
%

\section{Higher Leray--Serre spectral sequences in a simplicial framework}
\label{sec:serre_spectral_systems}

The first motivating example of Matschke's work~\cite{matschke2013successive} consists in higher spectral sequences defined from \emph{towers of fibrations}, that is, sequences of fibrations such that the total space of each is the base of the previous one:

\begin{equation}
\begin{tikzcd}
\label{eq:tower_of_fibrations} 
G_0 \arrow{r}{}  & E_0 \arrow{d}{} \\
\cdots  & \cdots \arrow{d}{} \\
G_{m-1} \arrow{r}{}  & E_{m-1} \arrow{d}{} \\
 & B
\end{tikzcd}
\end{equation}

In this situation, as the usual goal of computation is the homology $H_*(E_0)$ of the total space  of the upper fibration, one typically applies several times the Leray--Serre spectral sequence~\cite{Ser51}, assuming that the homology of $G_0, \ldots,G_{m-1}$ and $B$ is known. Leaving aside extension problems, one can think to determine $H_*(E_{m-1})$ from $H_*(B)$ and $H_*(G_{m-1})$ via a first Leray--Serre spectral sequence, using then a second Leray--Serre spectral sequence to try to determine $H_*(E_{m-2})$ from $H_*(E_{m-1})$ and $H_*(G_{m-2})$, and so on. A suitable higher spectral sequence, defined over the poset $D(\Z^m)$, represents a unified framework ``containing'' all these spectral sequences and offering a larger number of connections to the limit $H_*(E_0)$. Moreover, the $2$-page of the higher spectral sequence satisfies a formula which beautifully generalizes that of Serre for a fibration $G \rightarrow E \rightarrow B$, whose $2$-page can be expressed in terms of the homologies of the fiber $G$ and the base $B$ and converges to the homology of the total space $E$. Let us state here the result in the topological framework, as presented in \cite[Theorem~5.1]{matschke2013successive}; in Section \ref{sec:second_page} we will prove an analogous result in our simplicial framework.

\begin{theorem}
\label{thm:Serre-Matschke}
Consider a tower of $m$ fibrations of topological spaces, in the sense of Serre. There exists an associated higher spectral sequence over $D(\Z^m)$ with  $2$-page
\begin{equation*}
\label{eq:2-pageSerreMatschke}
S_n^*(P;m) \cong H_{p_m}(B;H_{p_{m-1}}(G_{m-1}; \ldots H_{p_{1}} (G_{1}; H_{p_0} (G_0)))),
\end{equation*}
with $P\=(p_1,\ldots ,p_{m})\in \Z^m$ and $p_0\= n-p_1 -\cdots -p_m$, which under suitable hypotheses (see \cite{matschke2013successive}) converges to $H_*(E_0)$.
\end{theorem}

As in the case of classical spectral sequences, this formula provides a description of an initial page of the higher spectral sequence; however, other terms $S[z,s,p,b]$ can only be determined in some simple cases. A first issue is that the differentials between terms of the 2-page, unlike the terms themselves, are not given explicitly. A second problem to determine the higher spectral sequence computationally arises when the involved chain complexes are not of finite type, as we have mentioned before, since standard algorithms for homology are bound to work with finite bases and matrices. However, the method we introduce in the present work is able to circumvent these obstacles by constructing the tower of fibrations in a simplicial framework and then using the effective homology of $E_0$ (which can be built automatically by Kenzo when $G_0, \ldots,G_{m-1}$ and $B$ are objects with effective homology) in order to determine the higher spectral sequence of the tower of fibrations by means of a higher spectral sequence associated with a chain complex of finite type.



In order to simplify the description of the results and make them more understandable, in Sections~\ref{sec:serre_spectral_systems} and~\ref{sec:efhm_for_serrespectralsystems} we consider the simple case of towers of two fibrations, as in diagram~(\ref{eq:tower-of-2-fibrations}). All the results we present carry over to the general case of $m$ fibrations, and in Section~\ref{sec:m_fibrations} we will provide a sketch of the proofs.

\begin{equation}
\label{eq:tower-of-2-fibrations}
\begin{tikzcd}
G \arrow{r}{}  & E  \arrow{d}{} \\
M \arrow{r}{}  & N \arrow{d}{} \\
 & B
\end{tikzcd}
\end{equation}

Let us suppose now that the two fibrations of diagram~(\ref{eq:tower-of-2-fibrations}) correspond to \emph{twisted Cartesian products}, introduced in Definition~\ref{def:twcrpr}. In other words, we assume that $G$ and $M$ are simplicial groups, $B$ is a simplicial set and $N$ and $E$ are defined as twisted Cartesian products $N \= M \times_{\tau_1} B$ and $E \= G \times_{\tau_0} N$, with   $\tau_1 : B \rightarrow M$ and $\tau_0 : N \rightarrow G $ the corresponding twisting operators. We choose this type of fibrations because they are the ones implemented in the Kenzo system and allow us to define generalized filtrations suitable for computations in the correct way. Furthermore, we assume $M$ and $B$ to be $1$-reduced (we explain the role of this hypothesis in Remark \ref{rmk:hyp1reduced}). Let us also remark that the base $B$ is not assumed
to be a Kan complex. In
contrast, the fibers $G$ and $M$  are modeled as simplicial groups,
so their underlying simplicial sets are in particular Kan complexes.

The construction of the higher Leray--Serre spectral sequence associated with (\ref{eq:tower-of-2-fibrations}) requires the definition of a filtration over the poset $D(\Zset^2)$ of the chain complex associated with $E \= G \times_{\tau_0} N \= G \times_{\tau_0} (M \times_{\tau_1}  B)$. 
The $D(\Z^2)$-filtration we consider for the chain complex $C_\ast(E)\= C_\ast(G \times_{\tau_0} (M \times_{\tau_1} B))$ is defined by means of the 
downsets $p(P;2)\= T^2_P$ introduced in Section~\ref{sec:SSys_over_downsets}. In fact, we define a filtration of the form $\{ F_{p(P;2)} \}_{P\in \Z^2}$, whose definition can be extended to produce a $D(\Z^2)$-filtration.

\begin{remark}
\label{rem:filtr_def}
In this section and in Section \ref{sec:efhm_for_serrespectralsystems} we will use the following strategy to define generalized filtrations on a chain complex $C_*$. 
\begin{itemize}
    \item At first we consider a grading over $\Z^2$ which allows us to define a collection $\{ C_P \}_{P\in \Z^2}$ of abelian groups. Each $C_P$ is generated by the elements of $C_*$ having \emph{(filtration) bidegree} $P=(p_1,p_2)\in \Z^2$ according to the considered grading.
    \item In the cases we will consider, the simple definition $F_p\= \bigoplus_{P\in p} C_P$ for $p\in D(\Z^2)$ does not provide a valid filtration of \emph{chain complexes}, because it does not satisfy $d(F_p)\subseteq F_p$ for all $p\in D(\Z^2)$. For example, simple downsets like $s(P)\= \{ Q\in \Z^2 \mid Q \le P \}$ may not fulfill $d(F_{s(P)})\subseteq F_{s(P)}$. We have therefore to consider a different definition.
    \item We use the downsets of the form $p(P;2)\= T^2_P$, introduced in Section~\ref{sec:SSys_over_downsets}, to define a filtration $\{ F_{p(P;2)} \}_{P\in \Z^2}$, where $F_{p(P;2)}\= \bigoplus_{X\in p(P;2)}C_X$. We have to prove that this provides a valid filtration of chain complexes, that is the condition $d(F_{p(P;2)})\subseteq F_{p(P;2)}$ is satisfied for all $p(P;2)$. Also, we show that the use of the downsets of the form $p(P;2)$ to define our filtrations is a \ql natural\qr\ choice.
    \item The definition of a filtration $\{ F_{p(P;2)} \}_{P\in \Z^2}$ can be extended to a filtration over $D(\Z^2)$ defining, for any $p\in D(\Z^2)$,
    \[ \overline{F}_p \= \{ \sigma \in C_* \mid \exists P \in \Z^2 \mbox{ such that } \sigma \in F_{p(P;2)} \mbox{ and } p(P;2) \subseteq p \}.
    \]
\end{itemize}
We introduce the extension of the filtrations $\{ F_{p(P;2)} \}_{P\in \Z^2}$ to $D(\Z^2)$-filtrations for the sake of generality, but the dependence on the downsets of the form $p(P;2)$ appears clear. In particular, as we mentioned, $\overline{F}_p$ is not equal to $\bigoplus_{X\in p}C_X$ in general. From now on in this work, we will define and consider filtrations of the form $\{ F_{p(P;2)} \}_{P\in \Z^2}$, as the extensions of all definitions and results to the case of generalized $D(\Z^2)$-filtrations $\{ \overline{F}_p \}$ is straightforward. 
\end{remark}



\begin{definition}
\label{defn:2crpr-filtration}
Let $E\= G \times_{\tau_0} (M \times_{\tau_1} B)$ as in diagram~(\ref{eq:tower-of-2-fibrations}) and $P \in \Z^2$. For the downset $p(P;2)\=T^2_{P}$ we define $F_{p(P;2)}$ as the chain subcomplex of $C_*(E)$ generated by the elements $\sigma\= (g,(m,b))$ in $E \= G \times_{\tau_0} (M \times_{\tau_1} B)$ such that $(x_1,x_2)\=(\deg(m,b)-\deg b, \deg b) \in p(P;2)$, where $\deg$ denotes the degeneracy degree introduced in Definition~\ref{def:deg_nondeg} and Remark~\ref{rmk:degCartesian}. 
\end{definition}


Notice that, as mentioned in Remark \ref{rem:filtr_def}, in Definition \ref{defn:2crpr-filtration} we introduce a grading over $\Z^2$, given by $(x_1,x_2)\=(\deg(m,b)-\deg b, \deg b)$, on the generators of $C_*(E)$, and we consider as generators of $F_{p(P;2)}$ only those satisfying $(x_1,x_2)\in p(P;2)$.

The reason why we decided to define a filtration on $C_\ast(E)$ in this way, which may not be intuitively clear, can be better understood considering the following pair of reductions:

\begin{equation}
\label{eq:2crpr-rdct}
\xymatrix @R=5.0mm{
  C_\ast(G \times_{\tau_0} (M \times_{\tau_1} B)) \ar @{=>>}_{\TEZ_0} [d]   \\
  C_*(G) \otimes_{t_0} C_* (M \times_{\tau_1} B) \ar @{=>>}_{\id \otimes_{t_0} \TEZ_1} [d]  \\
  C_* (G) \otimes_{t_0} (C_* (M) \otimes_{t_1} C_* (B)) &  
}
\end{equation}

The first reduction $\TEZ_0$ is the \emph{twisted Eilenberg--Zilber reduction} (Theorem~\ref{thm:twistedEZ}) of the fibration $G \rightarrow E \rightarrow N$; a reduction from $C_\ast(G \times_{\tau_0}(M\times_{\tau_1}B))$ to $C_\ast(G)\otimes_{t_0} C_\ast(M \times_{\tau_1} B)$ is obtained, where the symbol $\otimes_{t_0}$ represents a twisted (perturbed) tensor product, induced by the twisting operator $\tau_0$. Then, we consider
a new reduction given by the tensor product of the trivial reduction $\id$ of $C_*(G)$ 
and the twisted Eilenberg--Zilber reduction $\TEZ_1 :C_*(M\times_{\tau_1} B) \rrdc C_*(M)\otimes_{t_1} C_*(B)$, and using 
the Basic Perturbation Lemma (Theorem~\ref{thm:bpl}) (with the perturbation to be applied to the differential of $ C_*(G) \otimes C_* (M \times_{\tau_1} B)$ to obtain the
differential of $C_*(G) \otimes_{t_0} C_* (M \times_{\tau_1} B)$) we construct a reduction from $C_*(G) \otimes_{t_0} C_* (M \times_{\tau_1} B)$ to a new twisted
tensor product $C_* (G) \otimes_{t_0} (C_* (M) \otimes_{t_1} C_* (B))$. 

Let us remark that the bottom chain complex $C_* (G) \otimes_{t_0} (C_* (M) \otimes_{t_1} C_* (B)) $ in~(\ref{eq:2crpr-rdct}) has the same underlying graded module as the (non-twisted) tensor product $C_* (G) \otimes (C_* (M) \otimes C_* (B)) $ but the differential map has been perturbed due to the applications of the Basic Perturbation Lemma. The generators of degree $n$ of the chain complex $C_* (G) \otimes_{t_0} (C_* (M) \otimes_{t_1} C_* (B)) $ are elements of the form $\sigma\= g \otimes (m \otimes b)$ such that $g \in G_{p_0}$, $m \in M_{p_1}$ and $ b \in B_{p_2}$ for some $p_0,p_1,p_2 \geq 0$ with $p_0+p_1+p_2=n$. 
Then, it seems a natural choice to define a  filtration on $C_* (G) \otimes_{t_0} (C_* (M) \otimes_{t_1} C_* (B))$ over $D(\Zset^2)$ by using the points $P\=(p_1,p_2)$ to define a grading over $\Z^2$ on the generators. Keeping in mind Remark \ref{rem:filtr_def}, let us denote by $C_P$ the submodule of $C_* (G) \otimes_{t_0} (C_* (M) \otimes_{t_1} C_* (B))$ generated by its generators $\sigma$ of bidegree $P=(p_1,p_2)$. 
As we will detail below, the perturbed differential of $C_* (G) \otimes_{t_0} (C_* (M) \otimes_{t_1} C_* (B))$ applied to a generator of bidegree $P=(p_1,p_2)$ can increase the first or the second coordinate of this grading. 
This means that we cannot simply define a filtration setting $F_p\= \bigoplus_{P\in p} C_P$ for all $p\in D(\Z^2)$, because in general these submodules are not \emph{chain complexes} since they do not necessarily satisfy $d(F_p)\subseteq F_p$. As explained in Remark \ref{rem:filtr_def}, we solve this problem defining a filtration (of chain subcomplexes) of  $C_* (G) \otimes_{t_0} (C_* (M) \otimes_{t_1} C_* (B)) $ by using again the 
downsets $p(P;2)\=T^2_P$ introduced in Section~\ref{sec:SSys_over_downsets}.

\begin{definition}
\label{defn:2tnpr-filtration}
Let $C_\ast\= C_* (G) \otimes_{t_0} (C_* (M) \otimes_{t_1} C_* (B))$ obtained as in diagram~(\ref{eq:2crpr-rdct})  
and let $P\in \Z^2$. For the downset $p(P;2)\=T^2_{P}$ we define $F_{p(P;2)}$ as the chain subcomplex of $C_*$ generated by the elements $\sigma\= g \otimes (m \otimes b)\in C_n$ such that $g \in G_{x_0}, m \in M_{x_1}$ and $ b \in B_{x_2}$ for some $x_0,x_1,x_2 \geq 0$ with $x_0+x_1+x_2=n$ and $(x_1,x_2) \in p(P;2)$.
\end{definition}


In other words, we can state Definition \ref{defn:2tnpr-filtration} in the shorter form
\[ F_{p(P;2)}C_n\= \hspace{-10pt} \bigoplus_{\substack{i+j+k=n, \\ (j,k)\in p(P;2)}} \hspace{-10pt} 
 C_i (G) \otimes_{t_0} (C_j (M) \otimes_{t_1} C_k (B)).
\]

The relation between the chain complexes $C_\ast(E)\= C_\ast(G \times_{\tau_0} (M \times_{\tau_1} B))$ and $ C_* (G) \otimes_{t_0} (C_* (M) \otimes_{t_1} C_* (B))$ by means of the reductions of~(\ref{eq:2crpr-rdct}) motivated our idea to consider the filtration for $C_\ast(E)$ introduced in Definition~\ref{defn:2crpr-filtration}.

Before proving that the filtrations we defined are valid (that is, the submodules we defined are indeed chain subcomplexes), it is convenient to state a fact that will be often employed in the following proofs and arguments.

\begin{remark}
\label{rem:ineq_degrees}
Suppose we have integers $q'\le q$ and $p'\le p$. It clearly follows that  $(q'-p') -(q-p) \le p-p'$. If we define $p_1\= q-p$, $p_2\= p$, $x_1\= q'-p'$ and $x_2\= p'$, we easily see that the previous inequalities are equivalent to
\begin{align*}
  x_1 + x_2 &\le p_1 + p_2 ,\\ 
 x_2 &\le p_2  .
\end{align*}
Our typical application of this fact will be to maps (between chain complexes) reducing both coordinates of the grading we consider, corresponding here to $p$ and $q$, and it is useful if the parameters considered for filtering change to $p_1$ and $p_2$.
\end{remark}

\begin{proposition} \label{prop:Filtr_TwCar}
Definition~\ref{defn:2crpr-filtration} yields a valid filtration of chain subcomplexes. 
\end{proposition}
%
%
\begin{proof}

Consider $E\= G \times_{\tau_0} (M \times_{\tau_1} B)$ and the associated chain complex $C_*(E) \= C_*(G \times_{\tau_0} (M \times_{\tau_1} B))$. We have to prove that the differential is compatible with the defined filtration, that is $d (F_{p(P;2)})\subseteq F_{p(P;2)}$, for all $P\in \Z^2$.



Recall that the differential map $d_n :C_n (E) \to C_{n-1}(E)$ is defined as $d_n = \sum_{i=0}^n (-1)^{i} \partial_i$, where $\partial_i$ are the face operators of~$E_n$ (see Section~\ref{sec:simplicial_sets_and_fibrations}). 
Let $\sigma\=(g,(m,b))\in E_n$ such that $\sigma \in F_{ p(P;2)}$, which by Definition~\ref{defn:2crpr-filtration} means $(\deg(m,b)-\deg b, \deg b) \in p(P;2)$.
Let us denote $p\= \deg b$, $q\= \deg (m,b)$, and $p_1 \= q-p$, $p_2 \= p$. 
It follows from the definitions that, for every face operator $\partial_i$, if we denote $(g',(m',b'))\= \partial_i (g,(m,b))$ we have $p'\= \deg b' \le p$ and $q'\= \deg (m',b') \le q$. Using Remark~\ref{rem:ineq_degrees}, we see that $x_1\= q'-p'$ and $x_2\= p'$ must satisfy $x_1 + x_2\le p_1 +p_2$ and $x_2\le p_2$. It can be easily shown 
that all the points $(x_1,x_2)$ with $x_1 + x_2 \le p_1 +p_2$ and $x_2\le p_2$ belong to the downset  $p(P;2)=T^2_{P}$ (we graphically represented this fact in Figure~\ref{fig:T2_Filtr_TwCar}).
For example, we can observe that the points $(x_1,x_2)$ satisfy
\begin{equation}
\label{eq:xpv}
(x_1, x_2) = (p_1, p_2) + \lambda_1 v_1 + \lambda_2 v_2, \quad \text{for some } \lambda_1,\lambda_2 \in \Z_{\ge 0},
\end{equation}
with $v_1 \= -e_1 = (-1,0)$ and $v_2 \= e_1 -e_2= (1,-1)$, and show that the translation of $p(P;2)$ by $v_1$ (resp. $v_2$) is contained in $p(P;2)$. The case of $v_1$ is trivial, since $p(P;2)$ is a downset; in the case of $v_2$ we have
\[
p(P;2) + v_2 = p(P +v_2;2) = s(P;2)  = p(P;2) \setminus \{ P\} . 
\]
In conclusion, we have proven that ${\partial_i\sigma}\in F_{p(P;2)}$, for all $0 \leq i \leq n$.
\end{proof}
\begin{figure}[ht]
\centering
\includegraphics[scale=0.7]{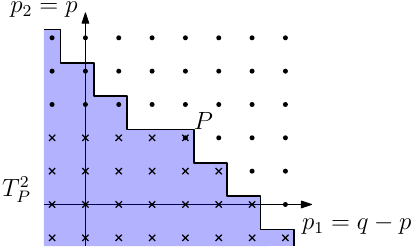}
\caption{Graphical representation of the proof of Proposition~\ref{prop:Filtr_TwCar}. Here we focus on the point $P\in T^2_{P}$. Starting from position $P$, the differential of $C_*(E)$ (see proof) can only reach the positions marked by the symbol {$\times$}. Recall that for a generator $(g,(m,b))$ of $C_n(E) \coloneqq C_n(G \times_{\tau_0} (M \times_{\tau_1} B))$ we have denoted $p\coloneqq \deg b$ and $q\coloneqq \deg (m,b)$.}
\label{fig:T2_Filtr_TwCar}
\end{figure}
%

Before showing the validity of the filtration we defined on the chain complex $C_* (G) \otimes_{t_0} (C_* (M) \otimes_{t_1} C_* (B))$, we state a useful result, proved in \cite[\S~8.3]{RS06}.
\begin{proposition} \label{prop:perturb_filtr}
Let $X,Y$ be simplicial sets and suppose we have a twisted Eilenberg--Zilber reduction $C_* (X\times_{\tau} Y) \rrdc C_* (X) \otimes_{t} C_*(Y)$. Denote $\delta$ the perturbation of the differential of $C_* (X\times_{\tau} Y)$, defined as in equation~(\ref{eq:pert_twist}), and denote $\delta'$ the induced perturbation of the differential of $C_* (X) \otimes_{t} C_*(Y)$. Then: 
\begin{enumerate}[(i)]
\item $\delta$ decreases (at least) by 1 the degeneracy degree of the second components (corresponding to simplices of $Y$), that is it sends a generator $(x,y)\in X_n \times Y_n$ with $\deg y =p$ to a linear combination of elements $(x_i,y_i)$, with $\deg y_i \le p-1$ for each $i$.
\item If we assume $Y$ to be 1-reduced, $\delta'$ decreases the dimension of the factor $C_*(Y)$ (at least) by 2, that is it sends $C_i(X) \otimes C_j (Y)$ to $\bigoplus_{k\ge 2}C_{i+k-1}(X) \otimes C_{j-k} (Y)$.
\end{enumerate}
\end{proposition}

\begin{proposition} \label{prop:F_tensor}
Definition~\ref{defn:2tnpr-filtration} yields a valid filtration of chain subcomplexes. 
\end{proposition}
\begin{proof}

We have to prove that the differential $d$ of the chain complex $C_* (G) \otimes_{t_0} (C_* (M) \otimes_{t_1} C_* (B))$ is compatible with the defined filtration, that is $d (F_{p(P;2)})\subseteq F_{p(P;2)}$, for all $P\in \Z^2$.



We denote by $p_1$ (dimension of the chain groups of $C_*(M)$) and $p_2$ (dimension of the chain groups of $C_*(B)$) the coordinates of the bidegree we used to define the filtration in Definition~\ref{defn:2tnpr-filtration}, and use the notation $D_* \= C_* (M) \otimes_{t_1} C_* (B)$. 

We can express the differential $d$ of $C_* (G) \otimes_{t_0} (C_* (M) \otimes_{t_1} C_* (B))$ as
\[ d = d^{\otimes}+d^2+d^3+\cdots ,
\]
where $d^{\otimes}$ denotes the differential of $C_*(G)\otimes D_*$ and the sum $d^2 +d^3 +\ldots$ represent its perturbation, with
\[ d^k_{p_0,j} : C_{p_0} (G) \otimes (C_* (M) \otimes C_* (B))_j \to C_{p_0 +k-1} (G) \otimes (C_* (M) \otimes C_* (B))_{j-k}
\]
for all $k\ge 2$, with $j =p_1 +p_2$. The fact that the perturbation concerns only indices $k\ge 2$ is a consequence of Proposition~\ref{prop:perturb_filtr} (ii), together with the assumption that $M$ and $B$ are $1$-reduced (in Section \ref{sec:efhm_for_serrespectralsystems} we present a detailed argument). As it can be easily shown using the definition of the downset $p(P;2)=T^2_{P}$, for each point $Q\=(q_1,q_2)\in p(P;2)$, all the points $(x_1,x_2)$ such that $x_1+x_2 =q_1+q_2-k$ (for $k\ge 2$) lie again in $p(P;2)$. That is, the perturbed part $d^2+d^3+\cdots$ of $d$ is compatible with the filtration we defined. 

We can then focus on $d^{\otimes} = d_{C(G)} \otimes \id_D \pm \id_{C(G)} \otimes d_D$, where the differential $d_D$ of $C_* (M) \otimes_{t_1} C_* (B)$
is given again by 
\[ d_D = \overline{d}^{\otimes} + \overline{d}^2 + \overline{d}^3 +\cdots ;
\]
$\overline{d}^{\otimes}$ is the differential of $C_* (M) \otimes C_* (B)$ and $\overline{d}^2 + \overline{d}^3 +\cdots$ is its perturbation, with
\[ \overline{d}^{\ell}_{p_1,p_2} : C_{p_1} (M) \otimes C_{p_2} (B) \to C_{p_1+\ell -1} (M) \otimes C_{p_2-\ell} .
\]
Similarly to before, it is easy to show that, for each point $Q\=(q_1,q_2)\in p(P;2)$, the points $(q_1-1,q_2),(q_1,q_2-1)$ and $(q_1+\ell -1,q_2-\ell)$, for each $\ell \ge 2$, lie again in $p(P;2)$. This completes the proof, whose idea is schematically represented in Figure~\ref{fig:T2_Filtr_Tensor}.
\end{proof}
\begin{figure}[ht]
\centering
\includegraphics[scale=0.7]{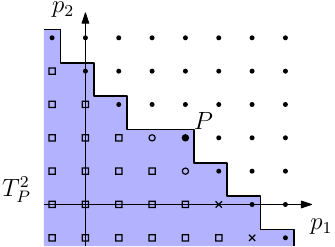}
\caption{Graphical representation of the proof of Proposition~\ref{prop:F_tensor}. Here we focus on the point $P\in T^2_{P}$. Starting from position $P$, the differential $d_D$ (see proof) can only reach the positions marked by {$\circ$} (differential $\bar{d}^{\otimes}$ of $C_*(M) \otimes C_*(B)$) and by {$\times$} (perturbations $\bar{d}^2, \bar{d}^3, \ldots$). As an effect of the perturbations $d^k$ of the differential $d$ (see proof), one possibly reaches the positions marked by {$\Box$}.}
\label{fig:T2_Filtr_Tensor}
\end{figure}
%

\section{Effective homology for computing higher Leray--Serre spectral sequences}
\label{sec:efhm_for_serrespectralsystems}

 The filtrations introduced in Definitions~\ref{defn:2crpr-filtration} and~\ref{defn:2tnpr-filtration} for the chain complexes $C_\ast(G \times_{\tau_0} (M \times_{\tau_1} B))$ and $C_* (G) \otimes_{t_0} (C_* (M) \otimes_{t_1} C_* (B))$ produce two higher spectral sequences that can be directly computed using the programs presented in Section~\ref{sec:programs} when the chain complexes are of finite type (which is true if and only if the three simplicial sets $G$, $M$ and $B$ are of finite type). 
However, when any of the simplicial sets $G$, $M$ and $B$ is not of finite type, the higher spectral sequences associated with the filtrations of $C_\ast(G \times_{\tau_0} (M \times_{\tau_1} B))$ and $C_* (G) \otimes_{t_0} (C_* (M) \otimes_{t_1} C_* (B))$ cannot be directly determined; in order to compute them, we have to resort to the effective homology technique, introduced in Section~\ref{sec:efhm}. In this section we detail how effective homology can be used to determine the higher spectral sequences also in the infinite type case.  
As in Section~\ref{sec:serre_spectral_systems}, for expository purposes we consider the case of towers of two fibrations; all the results we present will be generalized for $m$ fibrations in Section~\ref{sec:m_fibrations}.

Suppose that the simplicial sets $G$, $M$ and $B$ in the tower of two fibrations~(\ref{eq:tower-of-2-fibrations}) have effective homology, that is to say, there exist equivalences
\begin{equation}
\label{eq:BGequiv_gen_m2}
\begin{aligned}
C_* (G) \lrdc &\hat{C} G_* \rrdc  DG_* , \\
C_* (M) \lrdc &\hat{C} M_* \rrdc  DM_* , \\
C_* (B) \lrdc &\hat{C} B_* \rrdc  DB_* ,
\end{aligned}
\end{equation}
where $DG_*$, $DM_*$ and $DB_*$ are effective chain complexes.

Taking into account the tensor product of reductions and applying the Trivial Perturbation Lemma and the Basic Perturbation Lemma (Theorems~\ref{thm:tpl} and~\ref{thm:bpl} respectively), one obtains an equivalence given by the two reductions
\begin{align*} \rho'_1 : \quad & \hat{C} G_* \otimes_{t_0} (\hat{C} M_* \otimes_{t_1}\hat{C} B_*) \rrdc C_* (G) \otimes_{t_0} (C_* (M) \otimes_{t_1} C_* (B)) , \\
\rho'_2 : \quad & \hat{C} G_* \otimes_{t_0} (\hat{C} M_* \otimes_{t_1}\hat{C} B_*) \rrdc DG_* \otimes_{t_0} (DM_* \otimes_{t_1} DB_*) .
\end{align*}
Combining it with~(\ref{eq:2crpr-rdct}) we obtain the following diagram of reductions
connecting the chain complex $C_\ast(E)\= C_\ast(G \times_{\tau_0} (M \times_{\tau_1} B))$ to the effective chain complex $DG_* \otimes_{t_0} (DM_* \otimes_{t_1} DB_*)$:

\footnotesize
\begin{equation}
\label{eq:2crpr-efhm}
\xymatrix @R=5.0mm @C=0.0mm{
  C_\ast(G \times_{\tau_0} (M \times_{\tau_1} B)) \ar @{=>>}_{\TEZ_0} [d]  &  &  \\
  C_*(G) \otimes_{t_0} C_* (M \times_{\tau_1} B) \ar @{=>>}_{\id \otimes_{t_0} \TEZ_1} [d]  &  \hat{C} G_* \otimes_{t_0} (\hat{C} M_* \otimes_{t_1}\hat{C} B_*)  \ar @{=>>}_{\rho'_1} [ld] \ar @{=>>}^{\rho'_2} [rd] &  \\
  C_* (G) \otimes_{t_0} (C_* (M) \otimes_{t_1} C_* (B)) & & DG_* \otimes_{t_0} (DM_* \otimes_{t_1} DB_*) . 
}
\end{equation}
\normalsize

The reductions of this diagram allow us to use the effective homology technique to carry out computations on the filtered chain complex $C_\ast(G \times_{\tau_0} (M \times_{\tau_1} B))$ using the effective chain complex $DG_* \otimes_{t_0} (DM_* \otimes_{t_1} DB_*)$, on which an appropriate filtration will be defined. We want to show that, from the 2-page of the secondary connection, the effective homology method gives in fact correct results on the terms of the higher spectral sequence defined on $C_\ast(G \times_{\tau_0} (M \times_{\tau_1} B))$. For this purpose, we start from the top of the diagram and study the behavior of all the involved reductions. We will also define filtrations $\{ F_{p(P;2)} \}$ on all the involved chain complexes, recalling that Definitions~\ref{defn:2crpr-filtration} and~\ref{defn:2tnpr-filtration} already provide us with filtrations of $C_\ast(G \times_{\tau_0} (M \times_{\tau_1} B))$ and $C_* (G) \otimes_{t_0} (C_* (M) \otimes_{t_1} C_* (B))$. 
Remember that, as stated in  Section~\ref{sec:pert_EZ}, we always assume all the chain complexes associated with simplicial sets to be \emph{normalized}. Let us also recall that the simplicial sets $M$ and $B$ are assumed to be $1$-reduced, which clearly implies that also $M\times_{\tau_1} B$ is 1-reduced.

Let us start from the first reduction of the diagram, 
\[ \TEZ_0 : C_\ast(G \times_{\tau_0} (M \times_{\tau_1} B)) \rrdc  C_*(G) \otimes_{t_0} C_* (M \times_{\tau_1} B) ,
\]
which is a twisted Eilenberg--Zilber reduction (see Theorem~\ref{thm:twistedEZ}). 
We denote $\tilde{\delta}_0$ the perturbation induced via the twisted Eilenberg--Zilber theorem, so that we can express the differential of the chain complex $C_*(G) \otimes_{t_0} C_* (M \times_{\tau_1} B)$ as the differential of $C_*(G) \otimes C_* (M \times_{\tau_1} B)$ plus the perturbation $\tilde{\delta}_0$.
To study how it behaves with respect to filtrations, we first need to make explicit the filtration we consider on $C_*(G) \otimes_{t_0} C_* (M \times_{\tau_1} B)$.

\begin{definition}
\label{defn:crpr-tnpr-filtration}
Consider the chain complex $C_*(G) \otimes_{t_0} C_* (M \times_{\tau_1} B)$ and denote its generators by $\sigma\= g \otimes (m, b)$. Let $P\in \Z^2$. For the downset $p(P;2)\=T^2_{P}$ we define $F_{p(P;2)}$ as the chain subcomplex of $C_*(G) \otimes_{t_0} C_* (M \times_{\tau_1} B)$ generated by the elements $\sigma\= g \otimes (m, b)$ such that $(x_1,x_2)\=(\deg(m,b)-\deg b, \deg b) \in p(P;2)$.
\end{definition}

Notice that, since we are considering normalized chain complexes, saying $\deg (m,b)=q$ is equivalent to $(m,b) \in C_{q}(M \times_{\tau_1} B)$.

Once again, as described in Remark \ref{rem:filtr_def}, in Definition \ref{defn:crpr-tnpr-filtration} we introduce a grading over $\Z^2$, given by $(x_1,x_2)\=(\deg(m,b)-\deg b, \deg b)$, on the generators of $C_*(G) \otimes_{t_0} C_* (M \times_{\tau_1} B)$, and we consider as generators of $F_{p(P;2)}$ only those satisfying $(x_1,x_2)\in p(P;2)$.

The corresponding definitions of the filtrations introduced in Definitions~\ref{defn:2crpr-filtration}, \ref{defn:2tnpr-filtration} and~\ref{defn:crpr-tnpr-filtration} for towers of more than 2 fibrations are introduced in Definition~\ref{def:filtr_m_fibr}.

\begin{proposition}
Definition~\ref{defn:crpr-tnpr-filtration} yields a valid filtration of chain subcomplexes. 
\end{proposition}
\begin{proof} 
We have to prove that the differential of $C_*(G) \otimes_{t_0} C_* (M \times_{\tau_1} B)$ is compatible with the defined filtration.
Let $P\in \Z^2$ and consider the downset $p(P;2)\= T^2_P$. Proceeding in a similar fashion to the proof of  Proposition~\ref{prop:Filtr_TwCar}, we start by observing that the differential of $C_*(G) \otimes_{t_0} C_* (M \times_{\tau_1} B)$ does not increase the degeneracy degrees $p\= \deg b$ and $q\= \deg (m,b)$, as one can easily see considering the behavior of the differential of $C_*(G) \otimes C_* (M \times_{\tau_1} B)$ and of the perturbation  $\tilde{\delta}_0$, and recalling Proposition~\ref{prop:perturb_filtr}.

Since the bidegree we are considering is defined by $(q-p,p)$, by Remark~\ref{rem:ineq_degrees} the differential of $C_*(G) \otimes_{t_0} C_* (M \times_{\tau_1} B)$ sends a generator of bidegree $Q\=(q_1,q_2)\in p(P;2)$ to a linear combination of generators whose bidegrees $(x_1,x_2)$ satisfy $x_1+x_2\le q_1 +q_2$ and $x_2\le q_2$. Since all these points $(x_1,x_2)$ still belong to $p(P;2)$, we have shown that the differential is compatible with the defined filtration. Notice that Figure~\ref{fig:T2_Filtr_TwCar}, which is associated with Proposition \ref{prop:Filtr_TwCar}, depicts also the situation of this proof.
\end{proof}

We can now describe the behavior of the reduction $\TEZ_0$ with respect to the filtrations we are considering.

\begin{proposition}
\label{prop:rdct1-compatible}
The reduction $\TEZ_0$ in~(\ref{eq:2crpr-efhm}) is compatible with the filtrations defined on the chain complexes $C_\ast(G \times_{\tau_0} (M \times_{\tau_1} B))$ and $C_*(G) \otimes_{t_0} C_* (M \times_{\tau_1} B)$.
\end{proposition}

\begin{proof}
All the maps of the reduction $\TEZ_0 \=(f_{\TEZ_0},g_{\TEZ_0},h_{\TEZ_0})$, as it can be seen from their explicit definition (see Theorems~\ref{thm:EZ} and~\ref{thm:twistedEZ}), do not increase the degeneracy degrees $p\= \deg b$ and $q\= \deg (m,b)$. Therefore, since we use the bidegree given by $(q-p,p)$ to define the filtrations, we can again use Remark~\ref{rem:ineq_degrees} and the argument of the previous proof to conclude that $f_{\TEZ_0},g_{\TEZ_0},h_{\TEZ_0}$ are compatible with the considered filtrations. 
\end{proof}

Let us now turn our attention to the reduction
\[ \id \otimes_{t_0} \TEZ_1 : C_*(G) \otimes_{t_0} C_* (M \times_{\tau_1} B) \rrdc  C_*(G) \otimes_{t_0} ( C_* (M) \otimes_{t_1} C_* (B)) .
\]
The starting point of this reduction is the Eilenberg--Zilber reduction
\[ \TEZ_1 : C_* (M \times_{\tau_1} B) \rrdc   C_* (M) \otimes_{t_1} C_* (B) .
\]
Considering its tensor product with the trivial reduction of $C_* (G)$ we obtain
\[ \id \otimes \TEZ_1 : C_*(G) \otimes C_* (M \times_{\tau_1} B) \rrdc  C_*(G) \otimes ( C_* (M) \otimes_{t_1} C_* (B)) .
\]
Then, we introduce the perturbation $\tilde{\delta}_0$, which added to the differential of $C_*(G) \otimes C_* (M \times_{\tau_1} B)$ gives the differential of $C_*(G) \otimes_{t_0} C_* (M \times_{\tau_1} B)$, and apply the BPL to obtain the reduction $\id \otimes_{t_0} \TEZ_1$. Let us denote $\tilde{\delta}'_0$ the induced perturbation defining the differential of $C_*(G) \otimes_{t_0} ( C_* (M) \otimes_{t_1} C_* (B))$.

In order to apply the BPL, we have to make sure that the nilpotency condition is satisfied, as stated in Theorem~\ref{thm:bpl}. In this case, it is sufficient to observe that the composition $h_{\id \otimes \TEZ_1} \tilde{\delta}_0$ strictly decreases the degree $q\= \deg(m,b)$, since $h_{\id \otimes \TEZ_1}$ increases $q$ by $1$ (Proposition~\ref{prop:red_tensor}) and $\tilde{\delta}_0$ decreases $q$ at least by 2 (Proposition~\ref{prop:perturb_filtr} (ii)).

\begin{remark}
\label{rmk:hyp1reduced}
Let us clarify the role of the 1-reducedness assumption in (ii) of Proposition  \ref{prop:perturb_filtr}, which is used to prove the validity of some of the generalized filtrations of chain complexes we introduce (for example, we have used this assumption in Proposition \ref{prop:F_tensor}) and, even more importantly, as a sufficient condition for the nilpotency hypothesis of the BPL, to ensure that some reductions can be correctly defined. We have just used this argument for the reduction $\id \otimes_{t_0} \TEZ_1$ of (\ref{eq:2crpr-efhm}), and we will apply it again in this section for the reduction $\rho'_2$.
The 1-reducedness assumption can be relaxed by supposing that an explicit algebraic proof of the simply connectedness of the space is available, which guarantees that the thesis of Proposition  \ref{prop:perturb_filtr} (ii) holds. In Remark \ref{rmk:hyp1reduced_m} we provide further details. Note that, from a computational point of view, the 1-reducedness assumption has clear advantages, since it  can be easily verified. 
\end{remark}

We can now study the behavior of the reduction $\id \otimes_{t_0} \TEZ_1$ with respect to the filtrations we defined on $C_*(G) \otimes_{t_0} C_* (M \times_{\tau_1} B)$ and $C_*(G) \otimes_{t_0} ( C_* (M) \otimes_{t_1} C_* (B))$.

\begin{proposition}
\label{prop:rdct2-compatible}
Consider the reduction $\id \otimes_{t_0} \TEZ_1$ in~(\ref{eq:2crpr-efhm}) and the filtrations defined on the chain complexes $C_*(G) \otimes_{t_0} C_* (M \times_{\tau_1} B)$ and $C_* (G) \otimes_{t_0} (C_* (M) \otimes_{t_1} C_* (B))$. The maps $f_{\id \otimes_{t_0} \TEZ_1}$ and $g_{\id \otimes_{t_0} \TEZ_1}$ of the reduction are compatible with the filtrations, while $h_{\id \otimes_{t_0} \TEZ_1} (F_{p(P;2)}) \subseteq F_{p(P+(1,0);2)}$, for each $P\in \Z^2$.
\end{proposition}

\begin{proof}
The proof mainly consists in understanding the behavior of the maps of the reduction with respect to the gradings over $\Z^2$ we used to define the filtrations on $C_*(G) \otimes_{t_0} C_* (M \times_{\tau_1} B)$ and $C_* (G) \otimes_{t_0} (C_* (M) \otimes_{t_1} C_* (B))$. Recall that, for the chain complex $C_*(G) \otimes_{t_0} C_* (M \times_{\tau_1} B)$, we denote $p\= \deg b$ and $q\= \deg (m,b)$ for a generator $g\otimes (m,b)$, and we use the bidegree $(p_1,p_2)\=(q-p,p)$ to define the filtration; for the chain complex $C_*(G) \otimes_{t_0} C_* (M \times_{\tau_1} B)$, instead, we define directly the components $p_1$ (dimension of a chain in $C_*(M)$) and $p_2$ (dimension of a chain in $C_*(B)$) of the bidegree $(p_1,p_2)$ defining the filtration, and then we denote $p\= p_2$ and $q\=p_1+p_2$ to correctly keep track of all the indices.

Let us consider at first the maps of the reduction $\id \otimes \TEZ_1$. Using the definitions, it is easy to observe that the maps $f_{\id \otimes \TEZ_1}$ and $g_{\id \otimes \TEZ_1}$ do not increase the indices $p$ and $q$, while the map  $h_{\id \otimes \TEZ_1}$ does not increase the index $p$ and increases $q$ (at most) by 1. 
The operators
\begin{align*}
\varphi_{\id \otimes \TEZ_1}  \= \sum_{i=0}^{\infty}{(-1)^i(h_{\id \otimes \TEZ_1} \tilde{\delta}_0)^i} , \\
\psi_{\id \otimes \TEZ_1}  \= \sum_{i=0}^{\infty}{(-1)^i(\tilde{\delta}_0  h_{\id \otimes \TEZ_1})^i} ,
\end{align*}
of the BPL (see statement of Theorem~\ref{thm:bpl}) do not increase $p$ and $q$ as well, hence we can easily deduce the behavior of the maps of the reduction $\id \otimes_{t_0} \TEZ_1$: the maps $f_{\id \otimes_{t_0} \TEZ_1}$ and $g_{\id \otimes_{t_0} \TEZ_1}$ do not increase $p$ and $q$, while the map  $h_{\id \otimes_{t_0} \TEZ_1}$ does not increase the index $p$ and increases $q$ (at most) by 1. Remembering that we are using the bidegree $(p_1,p_2)\= (q-p,p)$ to define the filtrations, we can immediately conclude (recalling Remark~\ref{rem:ineq_degrees}) that $f_{\id \otimes_{t_0} \TEZ_1}$ and $g_{\id \otimes_{t_0} \TEZ_1}$ are compatible with the filtrations.  Since $h_{\id \otimes_{t_0} \TEZ_1}$ does not increase the index $p_2$ but can increase $p_1+p_2$ by $1$, Remark~\ref{rem:ineq_degrees} (with obvious modifications) tells us that $h_{\id \otimes_{t_0} \TEZ_1}$ sends a generator of filtration degrees $(p_1,p_2)$ to a linear combination of generators whose filtration degrees $(x_1,x_2)$ must satisfy $x_1+x_2\le p_1 +p_2 +1$ and $x_2\le p_2$ (see Figure~\ref{fig:hT2}). This implies that the image via $h_{\id \otimes_{t_0} \TEZ_1}$ of the chain subcomplex of the filtration indexed by the downset $p(P;2)=T^2_P$ is contained in the one indexed by  $p(P+(1,0);2)=T^2_{P+(1,0)}$.
\end{proof}

\begin{figure}[ht]
\centering
\includegraphics[scale=0.7]{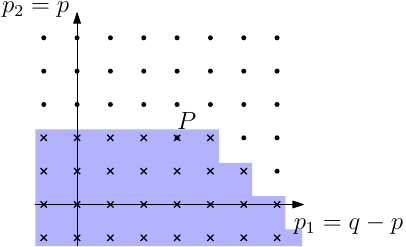}
\caption{Graphical representation of an argument of the proof of Proposition~\ref{prop:rdct2-compatible}. The homotopy $h_{\id \otimes_{t_0} \TEZ_1}$ applied to a generator having filtration bidegree $P\in \Z^2$ can only reach the positions marked by the symbol {$\times$}. In the figure we shaded the set of these points.}
\label{fig:hT2}
\end{figure}

We now go on to study the last two reductions $\rho'_1, \rho'_2$ of the diagram~(\ref{eq:2crpr-efhm}). 
First of all, it is convenient to set the notation for the reductions~(\ref{eq:BGequiv_gen_m2}) representing the effective homology of $G,M,B$. Let us denote 
\begin{align*} 
    \rho^{(0)}&\=(f^{(0)},g^{(0)},h^{(0)}): \hat{C}G_* \rrdc C_* (G) , \\
    \rho^{(1)}&\=(f^{(1)},g^{(1)},h^{(1)}): \hat{C}M_* \rrdc C_* (M) , \\
    \rho^{(2)}&\=(f^{(2)},g^{(2)},h^{(2)}): \hat{C}B_* \rrdc C_* (B) 
\end{align*}
and
\begin{align*} 
    \hat{\rho}^{(0)}&\=(\hat{f}^{(0)},\hat{g}^{(0)},\hat{h}^{(0)}): \hat{C}G_* \rrdc DG_* , \\
    \hat{\rho}^{(1)}&\=(\hat{f}^{(1)},\hat{g}^{(1)},\hat{h}^{(1)}): \hat{C}M_* \rrdc DM_* , \\
    \hat{\rho}^{(2)}&\=(\hat{f}^{(2)},\hat{g}^{(2)},\hat{h}^{(2)}): \hat{C}B_* \rrdc DB_* .
\end{align*}
Let us focus at first on the reduction
\begin{equation} \label{eq:rho'1} \rho'_1 :  \hat{C}G_* \otimes_{t_0} (\hat{C}M_* \otimes_{t_1} \hat{C}B_*) \rrdc C_*(G) \otimes_{t_0} (C_*(M) \otimes_{t_1} C_*(B)) . 
\end{equation}
In order to understand how its maps are defined, we consider how it is constructed. We start from the reduction 
\[ r \= \rho^{(1)}\otimes \rho^{(2)} : \hat{C}M_* \otimes \hat{C}B_* \rrdc C_*(M) \otimes C_*(B)  \]
and we perturb the differential of $C_*(M) \otimes C_*(B)$ by introducing the perturbation $\delta_1$ induced by the Eilenberg--Zilber reduction $\TEZ_1$. By Proposition~\ref{prop:perturb_filtr}, we know that $\delta_1$ reduces at least by $2$ the filtration degree $p$ (dimension of the chain groups of $C_*(B)$).
Now, applying the TPL, we obtain a reduction
\[ r_1 : \hat{C}M_* \otimes_{t_1} \hat{C}B_* \rrdc C_*(M) \otimes_{t_1} C_*(B) ,  \]
and Theorem~\ref{thm:tpl} gives us an explicit expression for the induced perturbation $\hat{\delta}_1$, from which one can easily observe that also $\hat{\delta}_1$ reduces at least by $2$ the degree $p$ (dimension of the chain groups of $\hat{C}B_*$).

Now we consider the reduction
\[ \overline{\rho}_1 \= \rho^{(0)}\otimes r_1 : \hat{C}G_* \otimes (\hat{C}M_* \otimes_{t_1} \hat{C}B_*) \rrdc C_*(G)\otimes (C_*(M) \otimes_{t_1} C_*(B))  \]
and we perturb the differential of $C_*(G)\otimes (C_*(M) \otimes_{t_1} C_*(B))$ by introducing the perturbation $\tilde{\delta}'_0$, that is the perturbation induced via the BPL (applied to construct the reduction $\id \otimes_{t_0} \TEZ_1$) by the perturbation $\tilde{\delta}_0$ (see above). In other words, the perturbation $\tilde{\delta}'_0$ is defined as 
\begin{equation} \label{eq:tilde'delta0}
    \tilde{\delta}'_0 \= f_{\id \otimes \TEZ_1} \tilde{\delta}_0 \varphi_{\id \otimes \TEZ_1} g_{\id \otimes \TEZ_1},
\end{equation}
and added to the differential of $C_*(G)\otimes (C_*(M) \otimes_{t_1} C_*(B))$ gives the differential of $C_*(G)\otimes_{t_0} (C_*(M) \otimes_{t_1} C_*(B))$. From the formula ~(\ref{eq:tilde'delta0}) we notice that $\tilde{\delta}'_0$ decreases at least by $2$ the degree $q$ (dimension of the chain groups of $C_*(M) \otimes_{t_1}C_*(B)$), since as mentioned before the maps $f_{\id \otimes \TEZ_1} , g_{\id \otimes \TEZ_1}$ and $\varphi_{\id \otimes \TEZ_1}$ do not increase the degree $q$. By applying the TPL we obtain the reduction $\rho'_1$ of~(\ref{eq:rho'1}). Let us denote $\hat{\delta}_0$ the perturbation induced by this application of the TPL, that is the perturbation which, added to the differential of $\hat{C} G_* \otimes (\hat{C} M_* \otimes_{t_1}\hat{C} B_*)$, yields the differential of $\hat{C} G_* \otimes_{t_0} (\hat{C} M_* \otimes_{t_1}\hat{C} B_*)$. Again, using its explicit definition (see Theorem~\ref{thm:tpl}), one can easily observe that  $\hat{\delta}_0$ reduces at least by $2$ the filtration degree $q$ (dimension of the chain groups of $\hat{C}M_* \otimes_{t_1}\hat{C}B_*$).

We can now show that, if we filter the chain complex $\hat{C} G_* \otimes_{t_0} (\hat{C} M_* \otimes_{t_1}\hat{C} B_*)$ mimicking the definition of the filtration of $C_*(G)\otimes_{t_0} (C_*(M) \otimes_{t_1} C_*(B))$, we obtain a valid filtration of chain complexes.

\begin{proposition}
\label{prop:tensor2_valid_filtr}
The filtration $\{ F_{p(P;2)} \}$ of the chain complex $\hat{C}_* \= \hat{C} G_* \otimes_{t_0} (\hat{C} M_* \otimes_{t_1}\hat{C} B_*)$ defined as in Definition~\ref{defn:2tnpr-filtration}, that is 
\[ F_{p(P;2)}\hat{C}_n\= \hspace{-10pt} \bigoplus_{\substack{i+j+k=n, \\ (j,k)\in p(P;2)}} \hspace{-10pt} 
 \hat{C} G_i \otimes_{t_0} (\hat{C} M_j \otimes_{t_1} \hat{C} B_k),
\]
is a valid filtration of chain subcomplexes.
\end{proposition}
\begin{proof}
The chain complex $\hat{C} G_* \otimes_{t_0} (\hat{C} M_* \otimes_{t_1}\hat{C} B_*)$ is obtained from $\hat{C} G_* \otimes (\hat{C} M_* \otimes\hat{C} B_*)$ introducing two perturbations: the perturbation $\hat{\delta}_1$, which added to the differential of $\hat{C} M_* \otimes \hat{C} B_*$ gives the differential of $\hat{C} M_* \otimes_{t_1}\hat{C} B_*$, and the perturbation $\hat{\delta}_0$, which added to the differential of $\hat{C} G_* \otimes (\hat{C} M_* \otimes_{t_1}\hat{C} B_*)$ gives the differential of $\hat{C} G_* \otimes_{t_0} (\hat{C} M_* \otimes_{t_1}\hat{C} B_*)$. 
We have showed that $\hat{\delta}_1$ decreases at least by $2$ the filtration degree $p$ (dimension of $\hat{C}B_*$) and $\hat{\delta}_0$ decreases at least by $2$ the filtration degree $q$ (dimension of $\hat{C}M_* \otimes_{t_1} \hat{C}B_*$). Since this is the same situation of the proof of Proposition~\ref{prop:F_tensor}, that argument carries over.
\end{proof}

The next result describes the behavior of the reduction $\rho'_1 =(f_{\rho'_1},g_{\rho'_1},h_{\rho'_1})$ of (\ref{eq:2crpr-efhm}) with respect to the defined filtrations. 
\begin{proposition} \label{prop:maps_rho'1}
The maps $f_{\rho'_1}$ and $g_{\rho'_1}$ of the reduction $\rho'_1$ are compatible with the defined filtrations. The homotopy $h_{\rho'_1}$ is such that
\[ h_{\rho'_1} (F_{p(P;2)}) \subseteq F_{p(P+(0,1);2)} ,
\]
for all $P\in \Z^2$.
\end{proposition}
\begin{proof}
The maps $f_{\rho'_1}$ and $g_{\rho'_1}$ preserve both degrees $p_1\= q-p$ (dimension of the chain groups of $\hat{C}M_*$, resp. $C_*(M)$) and $p_2\= p$ (dimension of the chain groups of $\hat{C}B_*$, resp. $C_*(B)$), as one can easily observe from their explicit definition (obtained by applying twice Proposition~\ref{prop:red_tensor}):
\[ f_{\rho'_1} \= f^{(0)}\otimes f^{(1)}\otimes f^{(2)}, \qquad g_{\rho'_1} \= g^{(0)}\otimes g^{(1)}\otimes g^{(2)} .
\]
The map
\[ h_{\rho'_1} \= h^{(0)}\otimes \id \otimes \id  + (g^{(0)}f^{(0)})\otimes h^{(1)} \otimes \id + (g^{(0)}f^{(0)})\otimes (g^{(1)} f^{(1)}) \otimes h^{(2)} 
\]
behaves differently: since the homotopies $h^{(0)}, h^{(1)} ,h^{(2)}$ are graded maps of degree $+1$, $h_{\rho'_1}$ sends a generator of $\hat{C} G_* \otimes_{t_0} (\hat{C} M_* \otimes_{t_1}\hat{C} B_*)$ of filtration bidegree $(p_1,p_2)$ to a linear combination of generators having filtration bidegrees $(p_1 +1,p_2)$ and $(p_1,p_2 +1)$. We obtain the thesis by observing that the points of any (fixed) downset $p(P;2)$ translated by $(1,0)$ or by $(0,1)$ are contained in the downset  $p(P+(0,1);2)$.
\end{proof}


We can now focus on the last reduction of (\ref{eq:2crpr-efhm}),
\begin{equation} \label{eq:rho'2} \rho'_2 :  \hat{C}G_* \otimes_{t_0} (\hat{C}M_* \otimes_{t_1} \hat{C}B_*) \rrdc DG_* \otimes_{t_0} (DM_* \otimes_{t_1} DB_*) .
\end{equation}
Once again, the best way to understand the behavior of its maps is to review how it is constructed. We start from the reduction 
\[ \hat{r} \= \hat{\rho}^{(1)}\otimes \hat{\rho}^{(2)} : \hat{C}M_* \otimes \hat{C}B_* \rrdc DM_* \otimes DB_*  \]
and we perturb the differential of $\hat{C}M_* \otimes \hat{C}B_*$ by introducing the perturbation $\hat{\delta}_1$ (see above).  Since $\hat{\delta}_1$ reduces at least by $2$ the filtration degree $p$ (dimension of the chain groups of $\hat{C}B_*$), the composition $h_{\hat{r}} \hat{\delta}_1$ strictly reduces the degree $p$, ensuring that the nilpotency condition of Theorem~\ref{thm:bpl} is satisfied.  
We can thus apply the BPL, obtaining a reduction
\[ \hat{r}_1 : \hat{C}M_* \otimes_{t_1} \hat{C}B_* \rrdc DM_* \otimes_{t_1} DB_* .  \]
Let us denote $\delta_1^D$ the  perturbation induced on $DM_* \otimes_{t_1} DB_*$  which, as it can be deduced from its explicit definition (see again the statement of Theorem~\ref{thm:bpl}), reduces at least by $2$ the degree $p$ (dimension of the chain groups of $DB_*$).

Now let us consider the reduction
\[ \overline{\rho}_2 \= \hat{\rho}^{(0)}\otimes \hat{r}_1 : \hat{C}G_* \otimes (\hat{C}M_* \otimes_{t_1} \hat{C}B_*) \rrdc DG_*\otimes (DM_* \otimes_{t_1} DB_*)  \]
and perturb the differential of $\hat{C}G_* \otimes (\hat{C}M_* \otimes_{t_1} \hat{C}B_*)$ by introducing the perturbation $\hat{\delta}_0$ defined above.
Since $\hat{\delta}_0$ decreases at least by $2$ the filtration degree $q$ (dimension of the chain groups of $\hat{C}M_* \otimes_{t_1} \hat{C}B_*$), the composition $h_{\overline{\rho}_2} \hat{\delta}_0$ strictly reduces the degree $p$, ensuring that the nilpotency condition of Theorem~\ref{thm:bpl} is satisfied. We can therefore apply the BPL,  obtaining the reduction $\rho'_2$ of~(\ref{eq:rho'2}). Let us denote $\delta_0^D$ the  perturbation induced by $\hat{\delta}_0$ on $DG_* \otimes (DM_* \otimes_{t_1} DB_*)$,  which again reduces at least by $2$ the filtration degree $q$ (dimension of the chain groups of $DM_* \otimes_{t_1} DB_*$).

As we did before for the chain complex $\hat{C} G_* \otimes_{t_0} (\hat{C} M_* \otimes_{t_1}\hat{C} B_*)$, we can now show that our usual definition of a filtration for twisted tensor products gives a valid filtration  of the chain complex $DG_* \otimes_{t_0} (DM_* \otimes_{t_1} DB_*)$.

\begin{proposition}
\label{prop:tensor3_valid_filtr}
The filtration $\{ F_{p(P;2)} \}$ of the chain complex $D_* \= DG_* \otimes_{t_0} (DM_* \otimes_{t_1} DB_*)$ defined as in Definition~\ref{defn:2tnpr-filtration}, that is 
\[ F_{p(P;2)}D_n\= \hspace{-10pt} \bigoplus_{\substack{i+j+k=n, \\ (j,k)\in p(P;2)}} \hspace{-10pt} 
 D G_i \otimes_{t_0} (D M_j \otimes_{t_1} D B_k),
\]
is a valid filtration of chain subcomplexes.
\end{proposition}
\begin{proof}
Similar considerations to the proof of Proposition~\ref{prop:tensor2_valid_filtr} hold.
The chain complex $DG_* \otimes_{t_0} (DM_* \otimes_{t_1} DB_*)$ is obtained from $DG_* \otimes (DM_* \otimes DB_*)$ introducing two perturbations: the perturbation $\delta_1^D$, which added to the differential of $DM_* \otimes DB_*$ gives the differential of $DM_* \otimes_{t_1} DB_*$, and the perturbation $\delta_0^D$, which added to the differential of $DG_* \otimes (DM_* \otimes_{t_1}DB_*)$ gives the differential of $DG_* \otimes_{t_0} (DM_* \otimes_{t_1}DB_*)$. 
We have showed that $\delta_1^D$ decreases at least by $2$ the filtration degree $p$ (dimension of $DB_*$) and $\delta_0^D$ decreases at least by $2$ the filtration degree $q$ (dimension of $DM_* \otimes_{t_1} DB_*$). Since this is the same situation of the proof of Proposition~\ref{prop:F_tensor}, that argument carries over.
\end{proof}

We can now study the behavior of the reduction $\rho'_2 =(f_{\rho'_2},g_{\rho'_2},h_{\rho'_2})$ with respect to the defined filtrations. As we will see, the role played by the BPL in defining the reduction $\rho'_2$ forces us to use slightly more subtle arguments than in Proposition~\ref{prop:maps_rho'1}.

\begin{proposition} \label{prop:maps_rho'2}
The maps $f_{\rho'_2}$ and $g_{\rho'_2}$ of the reduction $\rho'_2$ are compatible with the defined filtrations. The homotopy $h_{\rho'_2}$ is such that
\[ h_{\rho'_2} (F_{p(P;2)}) \subseteq F_{p(P+(0,1);2)} ,
\]
for all $P\in \Z^2$.
\end{proposition}
\begin{proof}
Recall that the maps $f_{\rho'_2}$ and $g_{\rho'_2}$ are defined as 
\begin{equation*}
f_{\rho'_2}  \= f_{\overline{\rho}_2}  \psi_{\overline{\rho}_2} , \qquad
g_{\rho'_2}  \= \varphi_{\overline{\rho}_2}  g_{\overline{\rho}_2} , 
\end{equation*}
where
\begin{equation} \label{eq:phi_psi_i}
\varphi_{\overline{\rho}_2}  \=  \sum_{i=0}^{\infty}{(-1)^i(h_{\overline{\rho}_2} \hat{\delta}_0)^i} , \qquad
\psi_{\overline{\rho}_2}  \= \sum_{i=0}^{\infty}{(-1)^i(\hat{\delta}_0  h_{\overline{\rho}_2})^i} .
\end{equation}
It is easy to see, as a direct consequence of their definitions, that the maps $f_{\overline{\rho}_2}$ and $g_{\overline{\rho}_2}$ do not change the filtration degree $q$ (dimension of the chain groups of $\hat{C}M_* \otimes_{t_1} \hat{C}B_*$, resp. $DM_* \otimes_{t_1} DB_*$) and do not increase the degree $p$  (dimension of the chain groups of $\hat{C}B_*$, resp. $DB_*$). Regarding the operators $\varphi_{\overline{\rho}_2}$ and $\psi_{\overline{\rho}_2}$, we can observe that they do not increase the filtration degree $q$, but we cannot say that they also maintain or decrease $p$. The reason for the last claim appears clear considering $\hat{\delta}_0 \= g_{\overline{\rho}_1} \tilde{\delta}'_0 f_{\overline{\rho}_1}$, where $\tilde{\delta}'_0$ is defined explicitly in equation~(\ref{eq:tilde'delta0}) and depends in turn on $\tilde{\delta}_0$, which can increase $p$ indefinitely. Nevertheless, we can distinguish the behavior of the summands with $i=0$ in equation~(\ref{eq:phi_psi_i}) from the behavior of the summands with $i>0$. The summands with $i=0$ in the equations defining $\varphi_{\overline{\rho}_2}$ and $\psi_{\overline{\rho}_2}$ are simply the identity of the chain complex $\hat{C}G_* \otimes_{t_0} (\hat{C}M_* \otimes_{t_1} \hat{C}B_*)$, which clearly maintains both $p$ and $q$ fixed. On the other hand, as we said, the summands with $i>0$ may increase the degree $p= p_2$, but they \emph{strictly} decrease the degree $q= p_1+p_2$ since $\hat{\delta}_0$ reduces $q$ at least by $2$. As a result, the operators $\varphi_{\overline{\rho}_2}$ and $\psi_{\overline{\rho}_2}$ send a generator of filtration bidegree $P\=(p_1,p_2)$ to a linear combination of generators whose filtration bidegrees can correspond \textquotedblleft exactly\textquotedblright\ to each point of the downset $p(P;2)=T^2_P$. This implies that the maps $f_{\rho'_2}$ and $g_{\rho'_2}$ are compatible with the defined filtrations $\{ F_{p(P;2)}\}$ of the chain complexes $\hat{C} G_* \otimes_{t_0} (\hat{C} M_* \otimes_{t_1}\hat{C} B_*)$ and  $DG_* \otimes_{t_0} (DM_* \otimes_{t_1} DB_*)$.

We now consider the map $h_{\rho'_2}  \= \varphi_{\overline{\rho}_2}  h_{\overline{\rho}_2}$ to prove the last part of the statement. Since we just described the behavior of  $\varphi_{\overline{\rho}_2}$ with respect to the degrees $p= p_2$ and $q= p_1+p_2$, we only need to focus on $h_{\overline{\rho}_2}$, which can increase both $p$ and $q$ at most by $1$, as one can easily observe from the expression
\[ h_{\overline{\rho}_2} \= \hat{h}^{(0)} \otimes \id_{\hat{C}M_* \otimes_{t_1}\hat{C}B_*} + (\hat{g}^{(0)} \hat{f}^{(0)}) \otimes h_{\hat{r}_1} .
\]
Therefore, because of the overall effect of $\varphi_{\overline{\rho}_2}$ and $h_{\overline{\rho}_2}$, the map $h_{\rho'_2}$ sends a generator of filtration bidegree $P\=(p_1,p_2)$ to a linear combination of generators whose filtration bidegrees correspond to the points of the downset $p(P;2)$ translated by $(0,1)$, that is the downset $p(P+(0,1);2)$. We display this behavior in Figure \ref{fig:h_rho2}. This implies that, considering the filtration $\{F_{p(P;2)}\}$ of $\hat{C} G_* \otimes_{t_0} (\hat{C} M_* \otimes_{t_1}\hat{C} B_*)$, the map $h_{\rho'_2}$ is such that
\[ h_{\rho'_2} (F_{p(P;2)}) \subseteq F_{p(P+(0,1);2)} ,
\]
for all $P\in \Z^2$.
\end{proof}

\begin{figure}[ht]
\centering
\includegraphics[scale=0.7]{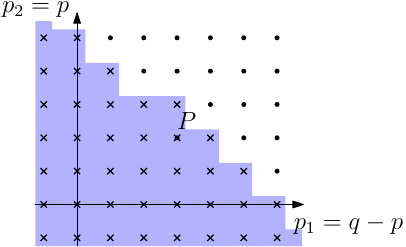}
\caption{Graphical representation of an argument of the proof of Proposition~\ref{prop:maps_rho'2}. The homotopy $h_{\rho'_2}$ applied to a generator having filtration bidegree $P\in \Z^2$ can reach the positions marked by the symbol {$\times$}. In the figure we shaded the set of these points, corresponding to the downset $p(P+(0,1);2)$.}
\label{fig:h_rho2}
\end{figure}


We can now prove the result which justifies the use of our effective homology methods to compute the higher Leray--Serre spectral sequence.

\begin{theorem}
\label{thm:2crpr-2page_S_effective}
From the $2$-page of the secondary connection (see Section~\ref{sec:SSys_over_downsets}), the higher Leray--Serre spectral sequence of the chain complex $C_\ast(G \times_{\tau_0} (M \times_{\tau_1} B))$ is naturally isomorphic to the higher spectral sequence we defined via the diagram~(\ref{eq:2crpr-efhm}) on the effective chain complex $DG_* \otimes_{t_0} (DM_* \otimes_{t_1} DB_*)$.
\end{theorem}
\begin{proof} 
Consider the diagram~(\ref{eq:2crpr-efhm}) and the filtrations $\{F_{p(P;2)}\}$ we defined on all its chain complexes.
We use the results of this section and Theorem~\ref{thm_Ana_orderred_gen} 
to show that corresponding terms of the $2$-pages of all the chain complexes involved in diagram~(\ref{eq:2crpr-efhm}) are isomorphic.

We start from the reduction $\TEZ_0$, which by Proposition~\ref{prop:rdct1-compatible} is compatible with the filtrations. Corollary~\ref{coro:thm_Ana_h_comp} then tells us that the higher spectral sequences associated with the filtered chain complexes $C_\ast(G \times_{\tau_0} (M \times_{\tau_1} B))$ and $C_*(G) \otimes_{t_0} C_* (M \times_{\tau_1} B)$ are isomorphic.

Consider now the reduction $\id \otimes_{t_0} \TEZ_1$, whose behavior is described in Proposition~\ref{prop:rdct2-compatible}. Recall that the terms $S^*(P;2)$ of the $2$-page of a higher spectral sequence are of the form $S[z,s,p,b]$ with
\begin{equation} \label{eq:recap-2page}
z \= T^2_{P+e_{1}-2e_2} , \quad
s\= T^2_{P+e_{1}-e_2}  , \quad  
p \= T^2_{P}  , \quad    
b \= T^2_{P+e_2} ,  
\end{equation}
where $e_1=(1,0)$ and $e_2=(0,1)$, see equation~(\ref{eq:TP_k}). In order to apply Theorem~\ref{thm_Ana_orderred_gen}, we have to remember the conditions~(\ref{eq.incl_gen_sameF}) on the homotopy $h$ of the reduction which guarantee we have isomorphic terms:
\begin{equation*} \label{eq.incl_gen_sameF2}
h(F_z) \subseteq F_s  \qquad \text{ and } \qquad
h(F_p)  \subseteq F_b   .
\end{equation*} 
From~(\ref{eq:recap-2page}) we see that for terms of the $2$-page the downset $s$ (resp. $b$) is a translation by $e_2=(0,1)$ of the downset $z$ (resp. $p$). We know from Proposition~\ref{prop:rdct2-compatible} that, for each $P\in \Zset^2$,
\begin{equation} \label{eq:double_incl}
    h_{\id \otimes_{t_0} \TEZ_1} (F_{p(P;2)}) \subseteq F_{p(P+(1,0);2)} \subseteq F_{p(P+(0,1);2)} ,
\end{equation}  
where the last inclusion is an immediate consequence of the \textquotedblleft shape\textquotedblright\ of the downsets we are considering. The homotopy $h_{\id \otimes_{t_0} \TEZ_1}$ satisfies therefore the conditions~(\ref{eq.incl_gen_sameF}) when $z,s,p,b$ are as in~(\ref{eq:recap-2page}), allowing us to conclude (by Theorem~\ref{thm_Ana_orderred_gen}) that the $2$-pages of the higher spectral sequences associated with $C_*(G) \otimes_{t_0} C_* (M \times_{\tau_1} B)$  and  $C_* (G) \otimes_{t_0} (C_* (M) \otimes_{t_1} C_* (B))$  are isomorphic.

Thanks to Proposition~\ref{prop:maps_rho'1} and Proposition~\ref{prop:maps_rho'2} we can apply the argument we just illustrated for $\id \otimes_{t_0} \TEZ_1$ to the reductions $\rho'_1$ and $\rho'_2$ and conclude that, from the $2$-page, the higher spectral sequences associated with the chain complexes
$C_* (G) \otimes_{t_0}( C_* (M) \otimes_{t_1} C_* (B))$, $\hat{C}G_* \otimes_{t_0} (\hat{C}M_* \otimes_{t_1} \hat{C}B_*)$ and $DG_* \otimes_{t_0} (DM_* \otimes_{t_1} DB_*)$ are isomorphic.

Let us finish the proof by remarking that the isomorphism between the different terms of the higher spectral sequences of both chain complexes $C_\ast(G \times_{\tau_0} (M \times_{\tau_1} B))$ and $DG_* \otimes_{t_0} (DM_* \otimes_{t_1} DB_*)$ is induced in a natural way by the composition of maps $f$ and $g$ of reductions of diagram~(\ref{eq:2crpr-efhm}).
\end{proof}

\begin{remark} \label{rem:mainthm}
In the proof of Theorem~\ref{thm:2crpr-2page_S_effective} we use the fact  (for example when we make use of Proposition~\ref{prop:maps_rho'1} and Proposition~\ref{prop:maps_rho'2}) that the homotopy $h$ of a certain reduction of the diagram~(\ref{eq:2crpr-efhm}) has a \textquotedblleft convenient\textquotedblright\ behavior, sending a chain subcomplex of the filtration indexed by a downset $p(P;2)$ to the chain subcomplex indexed by the same downset translated by $e_2=(0,1)$. This ensures that the conditions~(\ref{eq.incl_gen_sameF}) of Theorem~\ref{thm_Ana_orderred_gen} are satisfied for the terms $S[z,s,p,b]$ of the $2$-page, since we noticed from~(\ref{eq:recap-2page}) that the downset $s$ (resp. $b$) is a translation by $e_2=(0,1)$ of the downset $z$ (resp. $p$).
 In Figure~\ref{fig:15_2pageHomotopy2} we give an intuitive representation of this property of the $2$-page. Note that for the pages \ql preceding\qr\ the $2$-page in the secondary connection this is not true in general, as one can observe in Figure~\ref{fig:15_1h1pageHomotopy}.  
\end{remark}

\begin{remark} \label{rem:mainthm2}  
It is an easy consequence of the results we presented, relying on the conditions~(\ref{eq.incl_gen_sameF}) of Theorem~\ref{thm_Ana_orderred_gen}, that our method based on effective homology allows to compute also terms \ql following\qr\ the $2$-page, since by this we mean terms $S[z,s,p,b]$ such that the differences $s \setminus z$ and $b \setminus p$ are larger than for terms of the $2$-page (see \cite[\S~3]{matschke2013successive}).
\end{remark}

\begin{figure}[ht]
\centering
\includegraphics[scale=0.6]{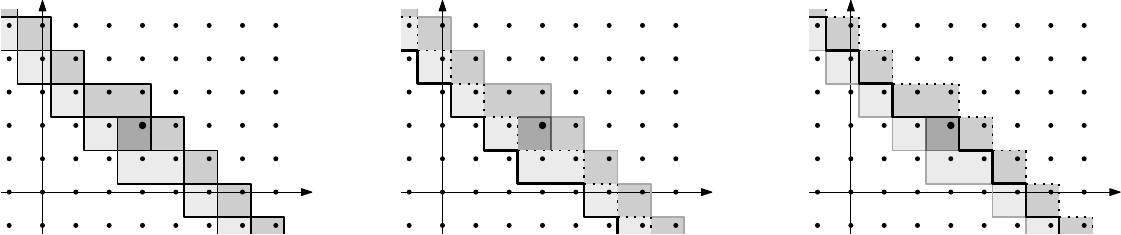}
\caption{The three figures represent the term $S^*(P;2)$ of the $2$-page of a higher spectral sequence over $D(\Z^2)$, with $P=(3,2)$. Left: representation of $S^*(P;2)$ with our usual conventions. Center: the downsets $z$ (black solid line) and $s$ (dotted line) are highlighted. Right: the downsets $p$ (solid) and $b$ (dotted) are highlighted.}
\label{fig:15_2pageHomotopy2}
\end{figure}
\begin{figure}[ht]
\centering
\includegraphics[scale=0.6]{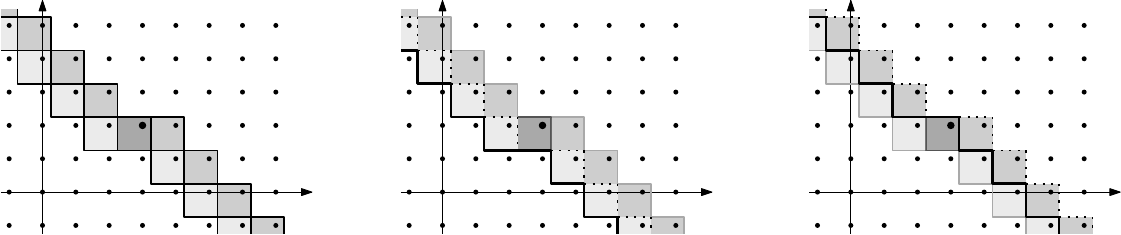}
\caption{Representations of the term $S(P;2)$ of a higher spectral sequence over $D(\Z^2)$, with $P=(3,2)$. Left: $S(P;2)$ represented with our usual conventions. Center: we can notice that the downset $z$ (black solid line) shifted by $(0,1)$ is not contained in the downset $s$ (dotted line). Right: the same situation holds for the downsets $p$ (solid) and $b$ (dotted).}
\label{fig:15_1h1pageHomotopy}
\end{figure}

The results presented in the current section allowed us to develop the following algorithm, which we implemented in the Kenzo system.

\begin{algorithm}
\label{alg:serre_2fibrations}
Computation of the higher Leray--Serre spectral sequence. \\
\emph{Input:}
\begin{itemize}
\item a tower of two fibrations $G\rightarrow E \rightarrow N$ and $M\rightarrow N \rightarrow B$ as in (\ref{eq:tower-of-2-fibrations}), defined by twisting operators $\tau_0: N \rightarrow G$ and $\tau_1:B \rightarrow  M$, where $M$ and $B$ are $1$-reduced.
\item equivalences as in (\ref{eq:BGequiv_gen_m2}), where $DG_\ast$, $DM_\ast$ and $DB_\ast$ are effective chain complexes.
\end{itemize}
\emph{Output:}
all the groups and differential maps of the higher Leray--Serre spectral sequence associated with the tower of fibrations, from the $2$-page of the secondary connection on.
\end{algorithm}

\begin{remark}
\label{rmk:higher_spseq_2}
As we mentioned in Section \ref{sec:SSys_over_downsets}, secondary connections are only one of the possible ways to connect the $1$-page to $H_*(E)$. The richer structure presented in \cite{Matschke2014a}, consisting in more terms and differentials for higher spectral sequences, provides a wide family of connections from the $1$-page to $H_*(E)$.
We conclude this section with a brief explanation of the following facts:
\begin{enumerate}[(i)]
\item Our simplicial construction of the higher Leray--Serre spectral sequence extends to the additional structure of \cite{Matschke2014a} for all terms and differentials following the $2$-page.
\item The technique of effective homology cannot be used to compute all terms and differentials of \cite{Matschke2014a}. The secondary connection we presented in Section \ref{sec:SSys_over_downsets}, based on downsets of the form $T^2_P$, appears to be ideal from this computational viewpoint.
\end{enumerate}
Despite our effort to make this work self-contained, for the sake of brevity we have to  refer the reader to \cite{Matschke2014a} for the notations and definitions we use in this remark.

To prove (i), recall that the simplicial version of the higher Leray--Serre spectral sequence rests on Definition \ref{defn:2crpr-filtration}, which introduces a filtration $\{ F_{p(P;2)} C_*(E)\}$ by assigning the bidegree  $X=(x_1,x_2)\=(\deg(m,b)-\deg b, \deg b)$ to the elements $(g,(m,b)) \in E \= G \times_{\tau_0} (M \times_{\tau_1} B)$ and setting $F_{p(P;2)} C_*(E)\= \bigoplus_{X\in p(P;2)} C_X$, where $C_X$ is generated by all elements of bidegree $X$. A term $S(P;\omega)$ of \cite{Matschke2014a} can be represented by downsets $z_{\omega},s_{\omega},p_{\omega},b_{\omega} \in D(\Z^2)$, meaning that $S(P;\omega) \cong S[z_{\omega},s_{\omega},p_{\omega},b_{\omega}]$, see \cite[Proof of Thm. 3.6]{Matschke2014a}. We show now that $F_{p_{\omega}}C_*(E) \= \bigoplus_{X\in p_{\omega}} C_X$ defines a chain subcomplexes of $C_*(E)$ (the case of $z_{\omega},s_{\omega},b_{\omega}$ is completely analogous) for each term  $S(P;\omega)$ following the $2$-page, which means $\omega = 12 * \tau$ with $\tau \in L^*_a$; together with the differentials naturally induced by the differential of $C_*(E)$, these terms constitute therefore a simplicial version of the additional structure of the higher Leray--Serre spectral sequence. Assuming $P=0$ to simplify the notations, the downset $p_{\omega}$ is defined by $p_{\omega}\= \{ X\in \Z^2 \mid M_{\omega} X \le_{\text{lex}} 0 \}$ (in this context, it is not restrictive to assume the permutation $\sigma$ involved in the definition of $p_{\omega}$ in \cite{Matschke2014a} to be the identity). From $\omega = 12 * \tau$ we obtain $M_{\omega} = M_{\tau} M_{12}$, where $M_{\omega} , M_{\tau} , M_{12} \in \Z^{2\times 2}_{\ge 0}$ (all entries of the three matrices are non-negative integers). Now look back at the proof of Proposition \ref{prop:Filtr_TwCar}. The differential of $C_*(E)$ sends an element $\sigma \= (g,(m,b))\in E$ of bidegree $(p_1,p_2)$ to a linear combination of elements of bidegree $(x_1,x_2)$ such that $x_1 + x_2 \le p_1 + p_2$ and $x_2 \le p_2$, which means that (\ref{eq:xpv}) holds: $(x_1,x_2)=(p_1,p_2) + \lambda_1 v_1 + \lambda_2 v_2$ for some $\lambda_1, \lambda_2 \in \Z_{\ge 0}$, with $v_1 \= -e_1$ and $v_2 \= e_1 -e_2$. Therefore, we only need to prove that the translation of $p_{\omega}$ by $v_i$ ($i=1,2$), given by $p_{\omega} + v_i \= \{ X\in \Z^2 \mid M_{\omega} X \le_{\text{lex}} M_{\omega} v_i \}$, is still contained in $p_{\omega}$.  As it can be easily shown (either directly or using the properties of $M_{\omega}$ described in \cite{Matschke2014a}), $M_{12} v_i \in \Z^2_{\le 0}$ for $i=1,2$. This implies that $M_{\omega} v_i \in \Z
^2_{\le 0}$, and in particular $M_{\omega} v_i \le_{\text{lex}} 0$, so we have $p_{\omega} + v_i \subseteq p_{\omega}$, for $i=1,2$.

To prove (ii), we provide a counterexample based on the fact that, in order to use effective homology, one must be able to define filtrations (of chain subcomplexes) on all the chain complexes of (\ref{eq:2crpr-efhm}). Here we focus on the chain complex $C_*(G)\otimes_{t_0} (C_*(M)\otimes_{t_1}C_*(B))$, filtered using the bidegree $(x_1,x_2)$ introduced in Definition \ref{defn:2tnpr-filtration}. Recall that the behavior of the differential of $C_*(G)\otimes_{t_0} (C_*(M)\otimes_{t_1}C_*(B))$ is summarized in Figure \ref{fig:T2_Filtr_Tensor}.  In Figure \ref{fig:S_omega1212} we show that for the word $\omega = 1212$ the downset $p_{\omega}$ is not closed with respect to the  differential of $C_*(G)\otimes_{t_0} (C_*(M)\otimes_{t_1}C_*(B))$.
\end{remark}
\begin{figure}[ht]
\centering
\includegraphics[scale=0.6]{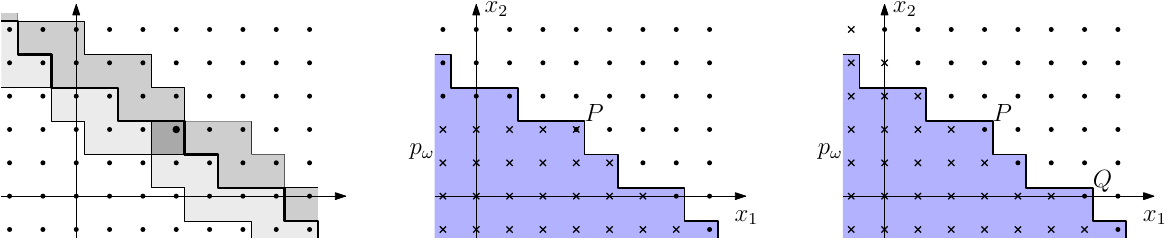}
\caption{Representation of the term $S(P;\omega)$, as defined in \cite{Matschke2014a}, with $P=(3,2)$ and $\omega = 1212$. Left: the term is represented as $S(P;\omega)\cong S[z_{\omega}, s_{\omega}, p_{\omega}, b_{\omega}]$, see Remark \ref{rmk:higher_spseq_2}. The downset $p_{\omega}$ is highlighted with a thicker line. Center: like in (i) of Remark \ref{rmk:higher_spseq_2}, we focus on the downset $p_{\omega}$ and the point $P$. Starting from position $P$, the differential of $C_*(E)$ can only reach the positions marked by $\times$ (compare with Figure \ref{fig:T2_Filtr_TwCar}), which still belong to $p_{\omega}$. This holds starting from each point of $p_{\omega}$. Right: as we claimed in (ii) of Remark \ref{rmk:higher_spseq_2}, starting from position $Q=(6,0)\in p_{\omega}$, the differential of $C_*(G)\otimes_{t_0} (C_*(M)\otimes_{t_1}C_*(B))$ can reach the positions marked by $\times$ (compare with Figure \ref{fig:T2_Filtr_Tensor}), some of which (for example $(0,4)$ and $(-1,5)$) do not belong to $p_{\omega}$.}
\label{fig:S_omega1212}
\end{figure}

\section[Sketch of results in the general case of \emph{m} fibrations]{Sketch of results in the general case of $m$ fibrations}
\label{sec:m_fibrations}

In this section we briefly explain the generalization of our results of Sections~\ref{sec:serre_spectral_systems} and~\ref{sec:efhm_for_serrespectralsystems} to the case of any finite number of fibrations. Let us consider again a diagram of~$m$ fibrations where the total space of each fibration is the base of the previous one:

\begin{equation}
\begin{tikzcd}
\label{eq:tower_of_fibrations2} 
G_0 \arrow{r}{}  & E_0 \arrow{d}{} \\
\cdots  & \cdots \arrow{d}{} \\
G_{m-1} \arrow{r}{}  & E_{m-1} \arrow{d}{} \\
 & B
\end{tikzcd}
\end{equation}

Like in the case of two fibrations, let us suppose that all the fibrations correspond to twisted Cartesian products, that is, $G_0, \ldots,  G_{m-1}$ are simplicial groups, $B$ is a simplicial set and each $E_i$ is defined as a twisted Cartesian product $E_i \= G_i \times_{\tau_i} E_{i+1}$, for all $0 \leq i \leq m-1$, with $\tau_i : E_{i+1} \rightarrow G_i$ and $E_m \= B$. We suppose also that the simplicial sets $G_1 ,\ldots ,G_{m-1},B$ are $1$-reduced and that $G_0,G_1 ,\ldots ,G_{m-1},B$ have effective homology, that is, there exist chain equivalences $C_* (G_i) \lrdc \hat{C} G_{i*} \rrdc  DG_{i*}$ for $0 \leq i \leq m-1$ and
$C_* (B) \lrdc \hat{C} B_* \rrdc  DB_*$ 
where all the $DG_{i*}$ and $DB_*$ are effective chain complexes. We remark that, as in the case of 2 fibrations, the base $B$ is not necessarily a Kan complex but the fibers $G_i$ are modeled as simplicial groups,
so their underlying simplicial sets are Kan complexes.

The effective homology of the top total space
\[ E_0 \= 
G_0 \times_{\tau_0} (G_1 \times_{\tau_1} \cdots \times_{\tau_{m-2}} (G_{m-1}\times_{\tau_{m-1}} B))
\]
is given by a composition of reductions that generalizes diagram (\ref{eq:2crpr-efhm}). To represent them in a diagram, let us introduce the following notations:
\begin{equation*}
\label{eq:notationC^i}
\begin{aligned}
C_*^{(0)} & \= C_*(G_0 \times_{\tau_0} (G_1 \times_{\tau_1} \cdots \times_{\tau_{m-2}} (G_{m-1}\times_{\tau_{m-1}} B))),     \\
C_*^{(1)} & \= C_*(G_0)  \otimes_{t_0} C_*(G_1 \times_{\tau_1} \cdots \times_{\tau_{m-2}} (G_{m-1}\times_{\tau_{m-1}} B)),   \\
C_*^{(2)} & \= C_*(G_{0}) \otimes_{t_0} (C_*(G_1)\otimes_{t_1} C_*(G_2 \times_{\tau_2} \cdots \times_{\tau_{m-1}} B)),    \\
 & \cdots    \\
C_*^{(m)} & \= C_* (G_0) \otimes_{t_0} (\cdots \otimes_{t_{m-2}} (C_* (G_{m-1})\otimes_{t_{m-1}} C_* (B))),    \\
\hat{C}_* & \=  \hat{C}G_{0*} \otimes_{t_0} (\cdots \otimes_{t_{m-2}} (\hat{C}G_{m-1,*} \otimes_{t_{m-1}} \hat{C}B_*)), \\
D_* & \= DG_{0*} \otimes_{t_0} (\cdots \otimes_{t_{m-2}} (DG_{m-1,*}\otimes_{t_{m-1}}  DB_*)). 
\end{aligned}
\end{equation*}
Then, the effective homology of $E_0$ is given by:
\begin{equation}
\label{eq:mcrpr-efhm}
\xymatrix @R=4.0mm @C=10.0mm{
C_*^{(0)}\ar @{=>>}_{\TEZ_0} [d]  &  &  \\
C_*^{(1)} \ar @{=>>}_{\id \otimes_{t_0} \TEZ_1} [d]  &  &  \\
C_*^{(2)}   \ar @{=>>}_{\id \otimes_{t_0} (\id \otimes_{t_1} \TEZ_2)} [d]  &  &  \\
\cdots \ar @{=>>}_{\id \otimes_{t_0} (\id \otimes_{t_1} \cdots \otimes_{t_{m-2}}\TEZ_{m-1})} [d]  &  \hat{C}_*  \ar @{=>>}_{\rho'_1} [ld] \ar @{=>>}^{\rho'_2} [rd] &  \\
   C_*^{(m)}  & & D_*
}
\end{equation}


The first reduction $\TEZ_0$ is simply the twisted Eilenberg--Zilber reduction (Theorem~\ref{thm:twistedEZ}) of the top fibration $G_0 \rightarrow E_0 \rightarrow E_1$. The second one $\id \otimes_{t_0} \TEZ_1$ is obtained by applying the Basic Perturbation Lemma (Theorem~\ref{thm:bpl}) to the tensor product (Proposition~\ref{prop:red_tensor}) of the trivial reduction of $C_\ast(G_0)$ and the twisted Eilenberg--Zilber reduction for the second fibration $G_1 \rightarrow E_1 \rightarrow E_2$. In general, the reductions of the type $\id \otimes_{t_0} \cdots \otimes_{t_{i-1}} \TEZ_i$ are defined by applying $i$ times the BPL (with the perturbations represented by $t_0, \ldots, t_{i-1}$) to the tensor product of the trivial reduction of $C_*(G_0)\otimes \cdots \otimes C_*(G_{i-1})$ and the twisted Eilenberg--Zilber reduction of the  fibration $G_i \rightarrow E_i \rightarrow E_{i+1}$. 
Lastly, the reduction $\rho'_1$ (resp. $\rho'_2$) is obtained by applying $m$ times the TPL (resp. BPL) to the tensor product of the $m+1$ reductions forming the equivalences that define the effective homology of $G_0, \ldots ,G_{m-1},B$.

\begin{remark}
\label{rmk:hyp1reduced_m}
As in the case of 2 fibrations (see Remark \ref{rmk:hyp1reduced}), the 1-reducedness assumption on $G_1 ,\ldots ,G_{m-1},B$ guarantees that we can apply the BPL and construct all the reductions in (\ref{eq:mcrpr-efhm}). After completing a first version of the present article, the authors discovered that towers of fibrations are studied in a simplicial setting in \cite{lambe1987applications}. The main result of that work corresponds, in our situation, to saying that $C^{(0)}_*$ and $D_*$ are chain equivalent, which is evident from the series of reductions in (\ref{eq:mcrpr-efhm}). The result is proven therein under weaker assumptions than in our situation, namely weaker conditions to ensure that the BPL can be applied repeatedly. We address the interested reader to the cited work, stressing once again the advantages of the 1-reducedness assumption from a computational perspective.
\end{remark}

In order to compute the higher Leray--Serre spectral sequence associated with $E_0 \= G_0 \times_{\tau_0} (G_1 \times_{\tau_1} \ldots \times_{\tau_{m-2}} (G_{m-1} \times_{\tau_{m-1}} B))$ we need to define generalized filtrations over the poset $D(\Zset^m)$ on all the chain complexes appearing in~(\ref{eq:mcrpr-efhm}), and we need to prove that all the reductions of that diagram have a \ql good\qr\ behavior with respect to the defined filtrations. To this aim, we use a natural generalization of Remark~\ref{rem:filtr_def} consisting in defining a \emph{multidegree} (over $\Z^m$) on the generators of each chain complex $C_*$, which produces a collection of abelian groups $\{C_P\}_{P \in \Z^m}$. Then, considering the downsets of the form $p(P;m)\= T^m_P$, a filtration $\{ F_{p(P;m)}\}_{P \in \Z^m}$ is defined as $F_{p(P;m)}\= \bigoplus_{X\in p(P;m)}C_X$. After proving that a filtration introduced in this way is compatible with the differential, that is, $d(F_{p(P;m)})\subseteq F_{p(P;m)}$, its definition can be extended to a filtration over $D(\Z^m)$ by setting, for any $p\in D(\Z^m)$,
\[ \overline{F}_p \= \{ \sigma \in C_* \mid \exists P \in \Z^m \mbox{ such that } \sigma \in F_{p(P;m)} \mbox{ and } p(P;m) \subseteq p \}.
\]

\begin{definition}
\label{def:filtr_m_fibr}
Consider the chain complex $C_\ast^{(i)}$ on the left in diagram~(\ref{eq:mcrpr-efhm}), that is
\[C_\ast^{(i)}\= C_*(G_{0}) \otimes_{t_0} (C_*(G_1)\otimes_{t_1} \cdots \otimes_{t_{i-1}} C_*(G_{i} \times_{\tau_i} \cdots \times_{\tau_{m-1}} B)) ,
\]
and let $\sigma \= g_0 \otimes( g_1 \otimes( \ldots \otimes (g_i,(g_{i+1}, \ldots, (g_{m-1}, b)))))$ be a generator of $C_\ast^{(i)}$. Denote $x'_m\=\deg(b)$; $x'_j\= \deg(g_j,(g_{j+1},\ldots,(g_{m-1},b))) $ for $i\leq j \leq m-1$; and $x'_j$ the integer such that $g_j \otimes( g_{j+1} \otimes( \ldots , b)) $ is a chain of dimension $x'_j$ in $C_{\ast}(G_j) \otimes_{t_j}( \cdots \times_{\tau_{m-1}} B))$ for $0 \leq j \leq i-1$. Then we set $x_m\=x'_m$ and $x_j\=x'_j-x'_{j+1}$ for $j\leq m-1$, we define the multidegree of $\sigma$ as  $X\=X_{\sigma}\=(x_1,\ldots , x_m)$ and we define $F_{p(P;m)}$ as the free $\Z$-module generated by the elements $\sigma$ of $C_\ast^{(i)}$ satisfying $X_{\sigma}\in p(P;m)$.
\end{definition}

The cases $i=0$ and $i=m$ in Definition \ref{def:filtr_m_fibr}, clearly corresponding  to the top and bottom-left chain complexes of diagram~(\ref{eq:mcrpr-efhm}), show that this definition represents a generalization of Definitions~\ref{defn:2crpr-filtration} and~\ref{defn:2tnpr-filtration}, as well as of Definition~\ref{defn:crpr-tnpr-filtration}. 
The proof that this definition produces valid filtrations is not included here, as it is an immediate generalization of the case of two fibrations we detailed in Sections~\ref{sec:serre_spectral_systems} and~\ref{sec:efhm_for_serrespectralsystems}. 
Moreover, the definition of the filtration for the particular bottom chain complex $C_* (G_0) \otimes_{t_0} (\cdots \otimes_{t_{m-2}} (C_* (G_{m-1})\otimes_{t_{m-1}} C_* (B)))$ can be easily adapted to the chain complexes  $\hat{C}G_{0*} \otimes_{t_0} (\cdots \otimes_{t_{m-1}} \hat{C}B_*)$ and $DG_{0*} \otimes_{t_0} (\cdots \otimes_{t_{m-1}}  DB_*)$.

\begin{remark}
\label{rmk:higher_spseq_m}
As in the case of towers of $2$ fibrations, the simplicial construction of the higher Leray--Serre spectral sequence for $m$ fibrations, based on the filtration we defined on $C^{(0)}_* = C_*(E_0)$, extends to the additional structure introduced in \cite{Matschke2014a} for all terms and differentials following the $2$-page. The argument of Remark \ref{rmk:higher_spseq_2} (i) can be easily generalized. 
\end{remark}

We studied the behavior of the reductions appearing in diagram~(\ref{eq:mcrpr-efhm}) with similar arguments to the ones presented in Section~\ref{sec:efhm_for_serrespectralsystems}. 
Once again, the proofs of the results about these reductions are an easy generalization of the case of two fibrations, only involving more complicated bookkeeping to describe how the  differentials and their perturbations modify the different filtration multidegrees. For this reason, as mentioned, we have chosen to detail the proofs only for towers of two fibrations, and to devote this section to state the general results we obtained for towers of $m$ fibrations. 

We have proved that the reductions on the left in diagram~(\ref{eq:mcrpr-efhm}) behave as follows. 

\begin{proposition}
\label{prop:rhoi_filtr}
Let $\rho^i\=(f^i,g^i,h^i)$ denote the $i$-th left reduction of diagram~(\ref{eq:mcrpr-efhm}), $\rho^i: C_*^{(i)} \rrdc C_*^{(i+1)}$. The maps $f^i, g^i$ and $h^i$ behave as follows with respect to the filtrations defined on the chain complexes $C_*^{(i)}$ and $C_*^{(i+1)}$: the maps $f^i$ and $g^i$ are compatible with the defined filtrations, while $h^i (F_{p(P;m)}) \subseteq F_{p(P+e_i;m)}$, for each $P\in \Z^m$ (with the convention $e_0\= 0$).
\end{proposition}

Regarding the reductions $\rho'_1$ and $\rho'_2$ in diagram~(\ref{eq:mcrpr-efhm}), we have proved the following results.

\begin{proposition}
\label{prop:rho1_filtr}
Let $\rho'_1\=(f'_1,g'_1,h'_1)$ be the reduction as in diagram~(\ref{eq:mcrpr-efhm}), 
\[ \rho'_1: \hat{C}_* \= \hat{C}G_{0*} \otimes_{t_0} (\cdots \otimes_{t_{m-1}} \hat{C}B_*)  \rrdc C^{(m)}_* \= C_* (G_0) \otimes_{t_0} (\cdots \otimes_{t_{m-1}} C_* (B)) .
\]
The maps $f'_1$, $g'_1$ and $h'_1$ behave as follows with respect to the filtrations defined on the chain complexes $\hat{C}_*$ and $C^{(m)}_*$: the maps $f'_1$ and $g'_1$ are compatible with the defined filtrations, while $h'_1 (F_{p(P;m)}) \subseteq F_{p(P+e_m;m)}$, for each $P\in \Z^m$.
\end{proposition}

\begin{proposition}
\label{prop:rho2_filtr}
Let $\rho'_2\=(f'_2,g'_2,h'_2)$ be the reduction  as in diagram~(\ref{eq:mcrpr-efhm}), 
\[ \rho'_2: \hat{C}_* \= \hat{C}G_{0*} \otimes_{t_0} (\cdots \otimes_{t_{m-1}} \hat{C}B_*)  \rrdc D_*\= DG_{0*} \otimes_{t_0} (\cdots \otimes_{t_{m-1}}  DB_*) .
\]
The maps $f'_2$, $g'_2$ and $h'_2$ behave as follows with respect to the filtrations defined on the chain complexes $\hat{C}_*$ and $D_*$: the maps $f'_2$ and $g'_2$ are compatible with the defined filtrations, while $h'_2 (F_{p(P;m)}) \subseteq F_{p(P+e_m;m)}$, for each $P\in \Zset^m$.
\end{proposition}

Now, recalling the definition of the downsets $p(P;k)\=T_P^k$ introduced in Section \ref{sec:SSys_over_downsets}, it is easy to observe that 
$$
p(P+e_i;m) \subseteq p(P+e_{i+1};m)\subseteq p(P+e_m;m), 
$$
for all $1 \leq i \leq m-1$.
Let us also recall that the terms $S^*(P;m)$ of the $2$-page of a higher spectral sequence are of the form $S[z,s,p,b]$, with
\begin{equation} \label{eq:recap-2page_m}
z \= T^m_{P+e_{m-1}-2e_m} , \quad
s\= T^m_{P+e_{m-1}-e_m}  , \quad  
p \= T^m_{P}  , \quad    
b \= T^m_{P+e_m}.  
\end{equation}
Therefore, the conditions of~Theorem~\ref{thm_Ana_orderred_gen}  guaranteeing that the terms of two spectral systems are isomorphic, that is $h(F_z) \subseteq F_s$ and $h(F_p) \subseteq F_b$, are satisfied for all the reductions of (\ref{eq:mcrpr-efhm}). In this way, we obtain the following general result which allows one to correctly apply the effective homology method to compute higher Leray--Serre spectral sequences of $m$ fibrations.

\begin{theorem}
\label{thm:mcrpr-2page_S_effective}
From the $2$-page of the secondary connection, the higher Leray--Serre spectral sequence of the chain complex $C_*(G_0 \times_{\tau_0} (G_1 \times_{\tau_1} \cdots \times_{\tau_{m-2}} (G_{m-1}\times_{\tau_{m-1}} B)))$ is isomorphic to the higher spectral sequence we defined via the diagram~(\ref{eq:mcrpr-efhm}) on the effective chain complex $DG_{0*} \otimes_{t_0} (DG_{1*} \otimes_{t_1}\cdots \otimes_{t_{m-2}} (DG_{m-1,*}\otimes_{t_{m-1}}  DB_*))$.
\end{theorem}

This leads to the following algorithm, which has been implemented in the Kenzo system.

\begin{algorithm}
\label{alg:mcrpr}
Computation of the higher Leray--Serre spectral sequence. \\
\emph{Input:}
\begin{itemize}
\item a tower of $m$ fibrations $G_i\rightarrow E_i \rightarrow E_{i+1}$ with $0\leq i \leq m-1$ as in (\ref{eq:tower_of_fibrations2}), defined by twisting operators $\tau_i: E_{i+1} \rightarrow G_i$, where $G_1,\ldots ,G_{m-1},B$ are $1$-reduced.
\item chain equivalences $C_* (G_i) \lrdc \hat{C} G_{i*} \rrdc  DG_{i*}$, for each $0 \leq i \leq m-1$, and $C_\ast(B) \lrdc \hat{C} B_\ast \rrdc DB_\ast$, where the $DG_{i\ast}$ and $DB_\ast$ are effective chain complexes.
\end{itemize}
\emph{Output:}
all the groups and differential maps of the higher Leray--Serre spectral sequence associated with the tower of fibrations, from the $2$-page of the secondary connection on.
\end{algorithm}

\section[Study of the 2-page]{Study of the $2$-page}
\label{sec:second_page}

We devote this section to the generalization of Serre's formula  for the $2$-page of the spectral sequence of a fibration to the context of higher spectral sequences associated with towers of $m$ fibrations. The formula is proved by Matschke \cite[Theorem~5.1]{matschke2013successive} in a topological framework, as we have seen in the statement of Theorem~\ref{thm:Serre-Matschke}. Here we present a different proof in the simplicial framework we have adopted throughout this paper. The theoretical results of this section provide a simplicial version of the higher Leray--Serre spectral sequence, in a similar way to what happened with the ordinary Leray--Serre spectral sequence, introduced by Jean Pierre Serre in 1951~\cite{Ser51} and translated to a simplicial language only in 1962 by Shih Weishu~\cite{Shih1962}. 

The result we will prove is the following:
\begin{theorem}
\label{thm:Serre-Matschke-simpl}
Consider a tower of $m$ fibrations in the simplicial framework we introduced, as in~(\ref{eq:tower_of_fibrations2}), and suppose that the simplicial sets $G_1, \ldots ,G_{m-1},B$ are $1$-reduced. The terms of the $2$-page of the associated higher spectral sequence are
\begin{equation}
\label{eq:2-pageSerreMatschke-simpl}
S^*_n(P;m)\cong H_{p_m}(B;H_{p_{m-1}}(G_{m-1}; \ldots H_{p_{1}} (G_{1}; H_{p_0} (G_0)))),
\end{equation}
with $P\=(p_1,\ldots ,p_{m})\in \Z^m$ and $p_0\= n-p_1 -\cdots -p_m$, and the higher spectral sequence converges to $H_*(E_0)$.
\end{theorem}

By virtue of the successive reductions of diagram~(\ref{eq:mcrpr-efhm}) between the chain complexes $C_* (G_0 \times_{\tau_0} (\cdots \times_{\tau_{m-2}} (G_{m-1} \times_{m-1} B)))$ and $C_* (G_0) \otimes_{t_0} (\cdots \otimes_{t_{m-2}} (C_* (G_{m-1})\otimes_{t_{m-1}} C_* (B)))$, we consider the latter for proving our results in this section. As a first step, we show that the formula~(\ref{eq:2-pageSerreMatschke-simpl}) holds for non-twisted tensor products $C_* (G_0) \otimes \cdots \otimes C_* (G_{m-1})\otimes C_* (B)$.
Then, we generalize the results to the case of $(m+1)$-fold twisted tensor products of the form $C_* (G_0) \otimes_{t_0} (\cdots \otimes_{t_{m-2}} (C_* (G_{m-1})\otimes_{t_{m-1}} C_* (B)))$, pointing out at last why the reductions~(\ref{eq:mcrpr-efhm}) guarantee that our results are valid also for the chain complex $C_* (G_0 \times_{\tau_0} (\cdots \times_{\tau_{m-2}} G_{m-1} \times_{m-1} B)))$.

Let us show that the expression~(\ref{eq:2-pageSerreMatschke-simpl}) for the terms of the $2$-page holds for non-twisted tensor products. 
Recall that we are considering the filtration of Definition \ref{def:filtr_m_fibr}, but in this case, instead of considering only the downsets $p(P;m)$, we extend the definition to all $p(P;k)$, for $1\le k\le m$. This yields a valid filtration of $C_* \= C_* (G_0) \otimes \cdots \otimes C_* (G_{m-1})\otimes C_* (B)$, as one can easily check, since the tensor product is not twisted. This filtration $\{ F_{p(P;k)}\}$ can be shortly defined as
\begin{equation}
\label{eq:def_FpPk_ntw}
F_{p(P;k)}C_n\= \hspace{-10pt} \bigoplus_{\substack{i_0 + i_1 + \ldots +i_m=n, \\ (i_1,\ldots ,i_m)\in p(P;k)}} \hspace{-10pt} 
 C_{i_0} (G_0) \otimes C_{i_1} (G_1) \otimes \cdots \otimes C_{i_{m-1}} (G_{m-1}) \otimes C_{i_m} (B),
\end{equation}
and clearly we have, as usual, a higher spectral sequence associated with it.
Even if the next result is included in~\cite{guidolin2018thesis}, we provide here a sketch of the proof for the sake of completeness.
\begin{proposition}
\label{prop:notwist-2page}
Let $C_* (G_0) \otimes \cdots \otimes C_* (G_{m-1})\otimes C_* (B)$ be a (non-twisted) tensor product of chain complexes; consider the filtration $\{ F_{p(P;k)}\}$ introduced in (\ref{eq:def_FpPk_ntw})
and the associated higher spectral sequence. Then the terms of the $2$-page can be expressed as in~(\ref{eq:2-pageSerreMatschke-simpl}).
\end{proposition}
\begin{proof}
Consider the homology groups of the form $H_n (F_{p(P;1)} /F_{s(P;1)})$, introduced in Section~\ref{sec:SSys_over_downsets} as the terms $S_n(P;1)$ forming the $1$-page of the higher spectral sequence, where $F_{p(P;1)} /F_{s(P;1)}$ denotes a subquotient of the chain complex $C_* (G_0) \otimes \cdots \otimes C_* (G_{m-1})\otimes C_* (B)$. Recall from (\ref{eq:TP_k}) that $s(P;1)\=p(P;1) \setminus \{ P\}$. As the set difference between the two posets $p(P;1)$  and $s(P;1)$ contains only $P\=(p_1,\ldots ,p_{m})\in \Z^m$, the homology group $H_n (F_{p(P;1)} /F_{s(P;1)})$ is isomorphic to the $n$-homology of
\[  C_{*-p_1 -\cdots -p_m}(G_0)\otimes C_{p_{1}}(G_1) \otimes \cdots \otimes C_{p_{m-1}}(G_{m-1})\otimes C_{p_m}(B),
\]
that is, by virtue of the \emph{Universal Coefficient Theorem} (see for example \cite[Theorem 11.1]{maclane1963homology}),
\begin{equation}
\label{eq:1pageS_homology}
H_{p_0}(G_0)\otimes C_{p_{1}}(G_1) \otimes \cdots \otimes C_{p_{m-1}}(G_{m-1})\otimes C_{p_m}(B), 
\end{equation}
with $p_0\= n-p_1 -\cdots -p_m$.
Recall now Lemma~\ref{lem:differentials-e_k} and note that taking homology in direction $-e_k$ corresponds to taking homology with respect to the differential $d_{G_k}\otimes \id \otimes \cdots \otimes \id$.
For example, the homology in direction $-e_1$ of~(\ref{eq:1pageS_homology}) is the $p_1$-homology of the chain complex
\[ H_{p_0}(G_0)\otimes C_{*}(G_1) \otimes C_{p_2}(G_2) \otimes \cdots \otimes C_{p_{m-1}}(G_{m-1})\otimes C_{p_m}(B),
\]
which is (applying again the Universal Coefficient Theorem)
\[  H_{p_1}(G_1 ;H_{p_0}(G_0)) \otimes C_{p_2}(G_2) \otimes \cdots \otimes C_{p_{m-1}}(G_{m-1})\otimes C_{p_m}(B) .
\]
It is now evident that iterating this argument $m$ times one obtains the formula~(\ref{eq:2-pageSerreMatschke-simpl}) for the terms of the $2$-page.
\end{proof}

Now we want to illustrate how the previous result can be used to study the case of twisted tensor products of chain complexes.

\begin{proposition}
\label{prop:delta}
Let $(C_*,d)$ and $(C_*,d')$ be two chain complexes having the same underlying graded abelian group $C_*$, but different differentials $d$ and $d'$.
Suppose furthermore that $\{ F_i\}_{i\in I}$ is an $I$-filtration (of chain complexes) for both $(C_*,d)$ and $(C_*,d')$, and denote the associated spectral systems with the letters $S$ and $S'$, respectively. Let $\delta \= d'-d$. If, for a $4$-tuple of indices $z\le s\le p\le b$ in $I$, the conditions
\begin{equation} \label{eq:cond_diff_isoterms}
\delta (F_p) \subseteq F_z   \qquad \text{ and } \qquad
\delta (F_b) \subseteq F_s   
\end{equation} 
hold, then $S[z,s,p,b]\cong S'[z,s,p,b]$.
\end{proposition}
\begin{proof}
It follows from the definition (\ref{eq:S_term_def1}) of the terms of a spectral system as quotient modules.
\end{proof}

Notice that in the conditions (\ref{eq:cond_diff_isoterms}) the involved indices are paired together in a different way from 
 (\ref{eq.incl_gen_sameF}) of Theorem \ref{thm_Ana_orderred_gen}.

The next result concerns filtrations of the form $\{ F_{p(P;m)}\}$ and terms of the type $S(P;m)$ and $S^*(P;m)$, defined as in~(\ref{eq:defS,S*}) of Section~\ref{sec:SSys_over_downsets}.

\begin{corollary}
\label{cor:isopages}
In the situation of Proposition \ref{prop:delta}, suppose that the filtration of both $(C_*,d)$ and $(C_*,d')$ is of the form $\{F_{p(P;m)}\}_{P\in \Z^m}$.
 If $\delta \= d'-d$ is such that
\begin{equation}
\label{eq:cond_diff_pPk}    
\delta (F_{p(P;m)}) \subseteq F_{p(P+e_{m-1}-2e_m;m)}
\end{equation}
for $P\=(p_1,\ldots ,p_m)\in \Z^m$, then we have isomorphisms $S^*(P;m)\cong S'^*(P;m)$ and $S(P;m)\cong S'(P;m)$.
\end{corollary}
\begin{proof}
The condition~(\ref{eq:cond_diff_pPk}) ensures that the inclusions~(\ref{eq:cond_diff_isoterms}) are satisfied for the indices of terms of the form $S^*(P;m)$ and $S(P;m)$, as one can easily see recalling (\ref{eq:TP_k}) and (\ref{eq:defS,S*}).
\end{proof}

Now we only need to prove that the two differentials of the non-twisted tensor product $C_* (G_0) \otimes \cdots \otimes C_* (G_{m-1})\otimes C_* (B)$ and of the twisted tensor product $C_* (G_0) \otimes_{t_0} (\cdots \otimes_{t_{m-2}} (C_* (G_{m-1})\otimes_{t_{m-1}} C_* (B)))$ behave as stated in the hypothesis of Corollary \ref{cor:isopages}. Notice that now we are considering a filtration of the form $\{ F_{p(P;m)}\}_{P\in \Z^m}$, since $\{ F_{p(P;k)}\}_{P\in \Z^m, 1\le k\le m}$ (as defined before in this section) in general is not a valid filtration for the twisted tensor product of chain complexes, as emerges from the results of the previous sections and more intuitively from Figure \ref{fig:T2_Filtr_Tensor} (for the case $m=2$).  

\begin{proposition}
\label{prop:isopages2}
Consider the chain complexes $C_* (G_0) \otimes \cdots \otimes C_* (G_{m-1})\otimes C_* (B)$ and $C_* (G_0) \otimes_{t_0} (\cdots \otimes_{t_{m-2}} (C_* (G_{m-1})\otimes_{t_{m-1}} C_* (B)))$, with the filtration $\{ F_{p(P;m)} \}$ defined before. Let $\delta \= d'-d$ denote the difference between the two differentials (that is, the perturbation associated with the $(m+1)$-fold twisted tensor product). Then~(\ref{eq:cond_diff_pPk}) holds, for all $P\=(p_1,\ldots ,p_m)\in \Z^m$; we have therefore isomorphisms $S^*(P;m)\cong S'^*(P;m)$ and $S(P;m)\cong S'(P;m)$, for all $P\=(p_1,\ldots ,p_m)\in \Z^m$.
\end{proposition}
\begin{proof}
The case for $m=2$ has already been detailed in the proof of Proposition~\ref{prop:F_tensor}; see also Figure \ref{fig:T2_Filtr_Tensor}. In order to prove the general case, it is sufficient to iterate the argument presented therein.
\end{proof}

\begin{remark}
Proposition~\ref{prop:isopages2} says that also for the terms of the page $\{S(P;m)\}$ immediately preceding the $2$-page $\{S^*(P;m)\}$ in the secondary connection (see Section \ref{sec:SSys_over_downsets}) we have isomorphisms when comparing the higher spectral sequences associated with the twisted and non-twisted tensor product of chain complexes. Nevertheless, we saw in Section \ref{sec:m_fibrations} that the vertical reductions of~(\ref{eq:mcrpr-efhm}) are well-behaved (that is, they give isomorphic terms) only from the $2$-page on. Therefore, keeping in mind that the higher spectral sequence we are considering is originally defined on the chain complex $C_*(G_0 \times_{\tau_0} ( \cdots \times_{\tau_{m-2}} (G_{m-1}\times_{\tau_{m-1}} B)))$, we cannot give an explicit formula for the terms of the page $\{S(P;m)\}$, but only for the $2$-page $\{S^*(P;m)\}$. 
\end{remark}

In summation, the results and considerations of the present section provide a proof of Theorem~\ref{thm:Serre-Matschke-simpl}, along with an explanation of the reason why the $2$-page is the only page of the secondary connection for which the presented arguments give an explicit formula.

\section{Examples and computations}
\label{sec:examples_and_computations}

The algorithms presented in Sections~\ref{sec:efhm_for_serrespectralsystems} and~\ref{sec:m_fibrations} have been implemented as a new module for the Kenzo system to compute higher Leray--Serre spectral sequences associated with towers of fibrations. The module enhances the previous module for computing generalized spectral sequences developed in~\cite{GR18}. All the programs are available at \url{https://github.com/ana-romero/Kenzo-external-modules}. 
In this section we present two different examples of application of our programs. 

\subsection{Whitehead tower of the sphere $S^3$}

As a first example of application of our programs for computing higher Leray--Serre spectral sequences of towers of fibrations, we consider the first stages of the {Whitehead} tower \cite{mccleary2001user,RS06,hatcher2002algebraic} for determining the homotopy groups of the sphere $S^3$, given by the following tower of fibrations: 

\begin{equation*}
\label{eq:s3-tower-of-2-fibrations}
\begin{tikzcd}
G\=K(\Zset_2,3) \arrow{r}{}  & E  \arrow{d}{} \\
M\=K(\Zset,2) \arrow{r}{}  & N \arrow{d}{} \\
 & B\=S^3
\end{tikzcd}
\end{equation*}

The sphere $S^3$ is represented in Kenzo as a simplicial set of finite type, with only two non-degenerate simplices: one simplex (the base point) in dimension~$0$ 
and another simplex in dimension $3$. In other words, the model chosen for the $3$-sphere can be seen as the quotient $\Delta^3 / \partial \Delta^3$. Let us observe that this model of $S^3$ is not a Kan complex. Eilenberg--MacLane spaces $K(\pi,n)$'s are represented in Kenzo by means of the classifying space constructor (see \cite{May67} for details). In particular, if the group $\pi$ is not finite (for instance $\Z$), then the set of $m$-simplices of $K(\pi,n)$ for every $m\geq n$ is infinite and hence  $K(\pi,n)$ is of infinite type. This model of Eilenberg--MacLane spaces 
provides simplicial groups, which satisfy the Kan property.

The simplicial set $N$ can be seen then as a twisted Cartesian product $N\= K(\Zset,2) \times_{\tau_1} S^3$ and the simplicial set $E$ as $E \= K(\Zset_2,3) \times_{\tau_0} N = K(\Zset_2,3) \times_{\tau_0} (K(\Zset,2) \times_{\tau_1} S^3)$, for twisting operators $\tau_0,\tau_1$ suitably defined~\cite{RS06}. One can separately consider the ordinary Leray--Serre spectral sequences of the two fibrations (respectively converging to the homology of $N$ and $E$), whose $2$-pages $\{ E^2_{p,q}\}$ are represented in Figure~\ref{fig:serre-ss1}. The groups of the first (left) Leray--Serre spectral sequence are given by the formula $E^2_{p,q} \cong H_p(S^3; H_q (K(\Z,2))$; those of the second (right) spectral sequence satisfy  $E^2_{p,q} \cong H_p(N ; H_q(K(\Z_2,3))$. For a description of the homology groups of $K(\Z_2,3)$, see \cite[\S 23]{EM54} or \cite[\S C.2]{Cle02}. Let us observe that, in the second spectral sequence, for total degree $n=p+q=5$ we obtain in this case only one non-null group $E^2_{0,5}\cong \Z_2$, which correspond to $\pi_5(S^3)\cong \Z_2$.

\footnotesize

\begin{figure}
\begin{tikzpicture}
\draw (0,0) node{
\begin{xy}<0.6cm,0cm>:<0cm,0.6cm>::
 (0,0)*{\hspace{0pt}} ;
 (0,0)*{\Z} ; (1,0)*{0} ; (2,0)*{0} ; (3,0)*{\Z} ; (4,0)*{0} ; (5,0)*{0} ; (6,0)*{0} ; (7,0)*{0} ;
 (0,1)*{0} ; (1,1)*{0} ; (2,1)*{0} ; (3,1)*{0} ;  (4,1)*{0} ; (5,1)*{0} ; (6,1)*{0} ;
 (0,2)*{\Z} ; (1,2)*{0} ; (2,2)*{0} ; (3,2)*{\Z} ; (4,2)*{0} ; (5,2)*{0} ;
 (0,3)*{0} ; (1,3)*{0} ; (2,3)*{0} ; (3,3)*{0} ; (4,3)*{0} ;
 (0,4)*{\Z} ; (1,4)*{0} ; (2,4)*{0} ; (3,4)*{\Z} ; 
 (0,5)*{0} ; (1,5)*{0} ; (2,5)*{0} ; 
 (0,6)*{\Z} ; (1,6)*{0} ; 
  (0,7)*{0 } ;
 (7.6,0.1) *!D{p} ; (0.1,7.6) *!L{q} ; 
 \ar@{.>} (0,0);(8,0)
 \ar@{.>} (0,0);(0,8)
 \end{xy}
};
\draw (6,0) node{
\begin{xy}<0.6cm,0cm>:<0cm,0.6cm>::
 (0,0)*{\hspace{0pt}} ;
 (0,0)*+{\Z} ; (1,0)*{0}     ; (2,0)*{0} ; (3,0)*{0} ; (4,0)*{\Zset_2} ;
 (5,0)*{0}   ; (6,0)*+{\Z_3} ; (7,0)*{\Z_2} ; 
 (0,1)*{0} ; (1,1)*{0} ; (2,1)*{0} ; (3,1)*{0} ;
 (4,1)*{0} ; (5,1)*{0} ; (6,1)*{0} ; 
 (0,2)*{0} ; (1,2)*{0} ; (2,2)*{0} ; (3,2)*{0} ;
 (4,2)*{0} ; (5,2)*{0} ; 
 (0,3)*{\Z_2} ; (1,3)*{0} ; (2,3)*{0} ; (3,3)*{0} ;
 (4,3)*{\Z_2} ; 
 (0,4)*{0} ; (1,4)*{0} ; (2,4)*{0} ; (3,4)*{0} ;
 %
 (0,5)*+{\Z_2} ; (1,5)*{0} ; (2,5)*{0} ;
 %
 (0,6)*+{\Z_2} ; (1,6)*{0} ; 
 %
  (0,7)*+{\Z_2} ;
 %
 %
 (7.6,0.1) *!D{p} ; (0.1,7.6) *!L{q} ; 
 \ar@{.>} (0,0);(8,0)
 \ar@{.>} (0,0);(0,8)
 \end{xy}
 };
 \end{tikzpicture}
 \caption{Pages $\{ E^2_{p,q}\}$ of the Leray--Serre spectral sequences of the fibrations $K(\Zset,2) \rightarrow N \rightarrow S^3$ (left) and $K(\Zset_2,3) \rightarrow E \rightarrow N$ (right), represented up to total degree $n=p+q=7$.}
\label{fig:serre-ss1}
 \end{figure}
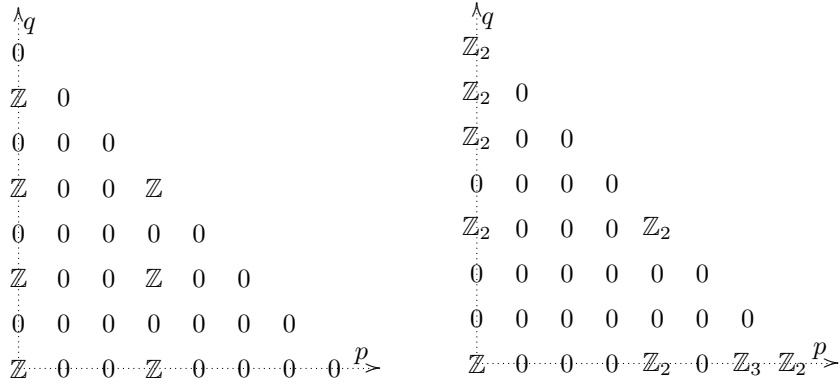

\normalsize

These fibrations are built in Kenzo with the following statements.

\footnotesize

\begin{verbatim}
> (progn
	(setf B (sphere 3))
	(setf k1 (chml-clss B 3))
	(setf t1 (z-whitehead B k1))
	(setf N (fibration-total t1))
	(setf k0 (chml-clss N 4))
	(setf t0 (z2-whitehead N k0))
	(setf E (fibration-total t0)))
[K298 Simplicial-Set]
\end{verbatim}

\normalsize
The result is the simplicial set \texttt{K298}, corresponding to the total space $E$. The simplicial set $E$ is not of finite type, since one of the factors of the twisted Cartesian product, namely $K(\Z,2)$, is of infinite type. Therefore, as we mentioned in Section \ref{sec:efhm}, it is not possible to apply standard algorithms based on matrix operations for computing homology groups, so one cannot directly compute the associated higher Leray--Serre spectral sequence. However, Kenzo automatically computes the effective homology of the object, given by an equivalence which is stored in the slot \texttt{efhm} and can be accessed as follows.

\footnotesize

\begin{verbatim}
> (efhm E)
[K608 Homotopy-Equivalence K298 <= K598 => K594]
\end{verbatim}

\normalsize

Now, in order to determine the associated higher Leray--Serre spectral sequence, we  filter the chain complex $C_*(E)$ canonically associated with the space $E$ by means of the generalized filtration introduced in Definition~\ref{defn:2crpr-filtration}. 

\footnotesize

\begin{verbatim}
>(setf Ef (change-chcm-to-gflcc E (dz2) crpr2-gflin 'crpr2-gflin))
[K611 Generalized-Filtered-Chain-Complex]
\end{verbatim}

\normalsize

In a similar way, the effective chain complex must be filtered by means of the generalized filtration introduced in Definition~\ref{defn:2tnpr-filtration}.

\footnotesize

\begin{verbatim}
> (setf Df (change-chcm-to-gflcc (rbcc (efhm e)) (dz2) tnpr2-gflin 'tnpr2-gflin))
[K613 Generalized-Filtered-Chain-Complex]
\end{verbatim}

\normalsize

By virtue of our results presented in Sections~\ref{sec:efhm_for_serrespectralsystems} and~\ref{sec:m_fibrations}, we know that all the corresponding terms of the higher spectral sequences of both chain complexes are isomorphic from the $2$-page on. In this way, using our algorithms we can determine the terms of the higher Leray--Serre spectral sequence of the total space $E$ (of infinite type) by computing the corresponding terms of the higher spectral sequence of the associated effective chain complex. To this aim, we use  the function \texttt{gen-spsq-group} (which receives the filtered chain complex, four elements $z\leq s \leq p \leq b$ in the poset of indices and the total degree $n$), and in particular the $2$-page $\{ S^*(P;2) \}$ can be determined in an easy way by means of the function \texttt{e2-gspsq-group} (which inputs the filtered chain complex, the point  $P = (p_1, p_2) \in \Z^2$ and the total degree~$n$). 

For instance, the groups of the $2$-page which are non-null for total degree $n=5$ are the following:

\footnotesize

\begin{verbatim}
> (e2-gspsq-group E '(0 0) 5)
Generalized spectral sequence S[((1 -2)),((1 -1)),((0 0)),((0 1) (1 0))]_{5}
Component Z/2Z
> (e2-gspsq-group E '(2 0) 5)
Generalized spectral sequence S[((0 0) (1 -1) (3 -2)),((0 1) (1 0) (3 -1)),
((0 1) (2 0)),((0 2) (2 1) (3 0))]_{5}
Component Z/2Z
> (e2-gspsq-group E '(2 3) 5)
Generalized spectral sequence S[((0 3) (1 2) (3 1) (4 0)),((0 4) (1 3) (3 2) (4 1)
(5 0)),((0 4) (2 3) (3 2) (4 1) (5 0)),((0 5) (2 4) (3 3) (4 2) (5 1) (6 0))]_{5}
Component Z
\end{verbatim}

\normalsize
Notice that the downsets $z,s,p,b$ indexing each term are stored as lists of points: each list of points in $\Z^2$ represents the smallest downset of $D(\Z^2)$ containing those points. Let us also observe that the obtained groups satisfy the formula $S^*_n(P;2)\cong H_{p_2}(B;H_{p_{1}}(M; H_{n-p_1-p_2} (G)))$ of Theorem \ref{thm:Serre-Matschke-simpl}.

As mentioned in Section~\ref{sec:SSys_over_downsets}, the \ql following\qr\ terms of the higher spectral sequence  allow to connect the $2$-page of the higher Leray--Serre spectral sequence to the homology $H_*(E)$ in many possible ways; in particular we have provided functions computing \emph{lexicographic connections}, whose terms have indices $z,s,p,b$ which can be defined by giving a point $P\in \Z^2$ and another parameter $Q\in \Z^2$, called \emph{offset} (see~\cite{matschke2013successive} for details). The function \texttt{lexcon-gspsq-group} inputs the filtered chain complex, the point $P=(p_1,p_2)\in \Z^2$, the offset $Q=(q_1,q_2) \in \Z^2$ and the total degree~$n$.

\footnotesize{
\begin{verbatim}
> (lexcon-gspsq-group E '(0 0) '(1 1) 5)
Generalized spectral sequence S[((0 -3)),((1 -1)),((0 0)),((1 2) (2 1) (3 0))]_{5}
Component Z/2Z
> (lexcon-gspsq-group E '(2 0) '(1 1) 5)
Generalized spectral sequence S[((0 -2) (2 -3)),((0 1) (1 0) (3 -1)),((0 1) (2 0)),
((0 4) (1 3) (3 2) (4 1) (5 0))]_{5}
NIL
> (lexcon-gspsq-group E '(2 3) '(1 1) 5)
Generalized spectral sequence S[((0 1) (2 0)),((0 4) (1 3) (3 2) (4 1) (5 0)),
((0 4) (2 3) (3 2) (4 1) (5 0)),((0 7) (1 6) (3 5) (4 4) (5 3) (6 2) (7 1) 
(8 0))]_{5}
NIL
\end{verbatim}
}

\normalsize

Let us observe that, although we have seen before that in the $2$-page of the higher spectral sequence there are three non-null terms of total degree $n=5$, in the next page, computed with offset $Q=(1,1)$, only one of these groups survives, namely $S^*_5(P;2)\cong H_0(B;H_0(M;H_5(G)))$, corresponding to the point $P=(0,0)$.

We can also determine the \ql final\qr\ groups of the higher spectral sequence, which intuitively are the ones that survive at the end of the lexicographic connection. In this case, for total degree $n=5$ we obtain exactly one non-null group, isomorphic to $\Z_2$,  which corresponds again to $S^*_5(P;2)\cong H_0(B;H_0(M;H_5(G)))$ and reproduces the well-known result $\pi_5(S^3)\cong \Z_2$.

\footnotesize

\begin{verbatim}
> (final-gspsq-group E 5)
Generalized spectral sequence S[((-1 -1)),((-1 -1)),((0 9) (1 8) (2 7) (3 6) (5 5)
(6 4) (7 3) (8 2) (9 1) (10 0)),((0 9) (1 8) (2 7) (3 6) (5 5) (6 4) (7 3) (8 2) 
(9 1) (10 0))]_{5}
Component Z/2Z
\end{verbatim}

\normalsize

The same programs can be applied to a larger number of fibrations. For instance, if we consider the next fibration of the {Whitehead} tower, namely $K(\Zset_2,4) \rightarrow X \rightarrow E$, we can determine for example the term of the $2$-page of the higher spectral sequence corresponding to $P=(0,6,0)$, which turns out to be $\Z$. We omit the  indices $z,s,p,b$ in the outputs.

\footnotesize

\begin{verbatim}
> (e2-gspsq-group X '(0 6 0) 6)
Generalized spectral sequence S[...]_{6}
Component Z
\end{verbatim}

\normalsize

However, the following term in the lexicographic connection with offset $Q=(1,1,1)$ is $\Z_{3}$; this represents a step of the convergence to the final group of the higher spectral sequence, which corresponds to the well-known result $\pi_6(S^3)\cong \Z_{12}$.

\footnotesize

\begin{verbatim}
> (lexcon-gspsq-group X '(0 6 0) '(1 1 1) 6)
Generalized spectral sequence S[...]_{6}
Component Z/3Z
> (final-gspsq-group X 6)
Generalized spectral sequence S[...]_{6}
Component Z/12Z
\end{verbatim}

\normalsize

As an indication of the efficiency of the algorithms, the last computation required 12 seconds on a normal laptop (Intel Core i5-8250U, 8Gb RAM). However, the algorithms producing the effective homology of a tower of fibrations in Kenzo have exponential complexity in the number of generators, so that the computation of the higher spectral sequences and homotopy groups are limited to relatively low dimensions (for instance, on the same laptop, Kenzo is able to determine the homotopy groups of $S^3$ up to dimension 7).

As a comparison with the computation of various subsequent spectral sequences, let us remark that the computation of all terms of the $2$-page of the higher spectral sequence for degree $n=6$ (84 groups) took 2 minutes and 14 seconds, while the computation of the groups $E^2_{p,q}$ of the three Leray--Serre spectral sequences corresponding to the three fibrations (21 groups) took only 6 seconds. This increase in computational cost for the higher spectral sequence is due to the fact that computing the groups of Definition~\ref{def:sp_sys_def1} is slower when the indices are elements in a poset like $\Z^m$ or $D(\Z^m)$, with $m>1$.

\subsection{Effective example}

Consider now a second example of tower of fibrations, given by the first stages of the Whitehead tower of the Eilenberg-MacLane space $K(\Zset_2,2)$:

\begin{equation*}
\label{eq:z2-tower-of-3-fibrations}
\begin{tikzcd}
G_0\= K(\Zset_2,5) \arrow{r}{}  & X_5  \arrow{d}{} \\
G_1\= K(\Zset_2,4) \arrow{r}{}  & X_4  \arrow{d}{} \\
G_2\= K(\Zset_2,3) \arrow{r}{}  & X_3 \arrow{d}{} \\
 & B\= K(\Zset_2,2)
\end{tikzcd}
\end{equation*}

In this case all the associated chain complexes are of finite type (effective), so that we can determine the higher spectral sequence without using the effective homology method. However, due to efficiency problems, it is better to use the effective homology of the space $X_5 \cong K(\Zset_2,5)\times_{\tau_0}(K(\Zset_2,4)\times_{\tau_1}(K(\Zset_2,3)\times_{\tau_2}K(\Zset_2,2)))$, which is obtained via the composition of several reductions as in diagram~(\ref{eq:mcrpr-efhm}).

In accordance with Theorem~\ref{thm:2crpr-2page_S_effective}, we have used our programs to verify (up to total degree $n=3$) that the terms of the $2$-page $\{ S_n^*(P;3)\}$ computed with the two methods coincide. 
The only non-null terms we obtained (with both methods) are the following:

{ 
\footnotesize 
\begin{verbatim}
(0 0 0) Generalized spectral sequence S[NIL,NIL,((0 0 0)),((0 0 1) (0 1 0)
(1 0 0))]_{0}
Component Z

(0 0 2) Generalized spectral sequence S[((0 1 0) (1 0 0)),((0 1 1) (0 2 0)
(1 0 1) (1 1 0) (2 0 0)),((0 0 2) (0 1 1) (0 2 0) (1 0 1) (1 1 0) (2 0 0)),
((0 0 3) (0 1 2) (0 2 1) (0 3 0) (1 0 2) (1 1 1) (1 2 0) (2 0 1) (2 1 0)
(3 0 0))]_{2}
Component Z/2Z

(0 3 0) Generalized spectral sequence S[((1 0 1) (1 1 0) (2 0 0)),((1 0 2)
(1 1 1) (1 2 0) (2 0 1) (2 1 0) (3 0 0)),((0 3 0) (1 0 2) (1 1 1) (1 2 0)
(2 0 1) (2 1 0) (3 0 0)),((0 3 1) (0 4 0) (1 0 3) (1 1 2) (1 2 1) (1 3 0)
(2 0 2) (2 1 1) (2 2 0) (3 0 1) (3 1 0) (4 0 0))]_{3}
Component Z/2Z
\end{verbatim}
}
\noindent where the triple of integers at the beginning of a line represents the point $P=(p_1,p_2,p_3)\in \Z^3$ for which we are computing the term $S^*_n (P;3)$.

Using effective homology we were able to efficiently compute the remaining non-null terms up to total degree $n=5$ (we omit again the four indices $z,s,p,b$ in the outputs):

{ 
\footnotesize 
\begin{verbatim}
(0 0 4) Generalized spectral sequence S[...]_{4}
Component Z/4Z

(4 0 0) Generalized spectral sequence S[...]_{4}
Component Z/2Z

(0 0 0) Generalized spectral sequence S[...]_{5}
Component Z/2Z

(0 0 5) Generalized spectral sequence S[...]_{5}
Component Z/2Z

(0 3 2) Generalized spectral sequence S[...]_{5}
Component Z/2Z

(0 5 0) Generalized spectral sequence S[...]_{5}
Component Z/2Z
\end{verbatim}
}

We can now show with an example that, for the terms of the pages preceding the $2$-page, one does not necessarily obtain isomorphic results using the two methods. For example, we found out that the term $S_3(P;1)$ with $P=(0,0,3)$ of the $1$-page is different if computed directly

{ 
\footnotesize 
\begin{verbatim}
> (e1-eff-gspsq-group X5 '(0 0 3) 3)
Generalized spectral sequence S[...]_{3}
Component Z
Component Z
Component Z
Component Z
Component Z
Component Z
Component Z
Component Z
\end{verbatim}
}

\noindent or using effective homology

{ 
\footnotesize 
\begin{verbatim}
> (e1-eff-gspsq-group effX5 '(0 0 3) 3)
Generalized spectral sequence S[...]_{3}
Component Z
\end{verbatim}
}


\section{Conclusions}
\label{sec:conclusions}

In this work we introduced specific methods to compute the higher Leray--Serre spectral sequence associated with a tower of fibrations. The programs we developed allow to determine terms and differential maps of a higher Leray--Serre spectral sequence, and have been released as a new module for the Computer Algebra system Kenzo. By virtue of the effective homology technique, our programs are able to handle objects of infinite type, and can therefore be employed in a wide variety of situations arising in algebraic topology. A first fundamental step towards making the higher Leray--Serre spectral sequence computable via algorithms was to rephrase its construction in a simplicial framework; in this paper, we include some relevant theoretical results on this simplicial version of the construction.   

We are convinced that some of the methods we introduced in this work can be adapted to compute other relevant spectral systems and higher spectral sequences described in \cite{matschke2013successive}, such as the Adams--Novikov and Eilenberg--Moore spectral systems or interesting instances of the higher Grothendieck spectral sequence, and we believe this represents an interesting direction for future research.

\section*{Acknowledgements}

We want to thank Francis Sergeraert for helpful discussions at an early stage of the work.

Supported by the Basque Government through the BERC 2018-2021 program and by the Spanish Ministry of Science, Innovation and Universities, BCAM Severo Ochoa accreditation SEV-2017-0718. Partially supported by the Spanish Ministry of Science, Innovation and Universities, project MTM2017-88804-P.


\bibliographystyle{plain}

\end{document}